\documentclass{amsart}
\usepackage[utf8x]{inputenc}
\usepackage{amsmath, amsthm, amssymb}
\usepackage{color}
\usepackage[margin=1.45in]{geometry}
\newtheorem{theorem}{Theorem}[section]
\newtheorem{lemma}{Lemma}[section]
\newtheorem{proposition}{Proposition}[section]
\newcommand{\R}{\mathbb R}
\newcommand{\eps}{\varepsilon}

\newcommand{\vp}{\varphi}
\newcommand{\sech}{\mbox{\textup{sech}}}
\newcommand{\dd}{\, \mathrm{d}}
\newcommand{\dst}{\mbox{dist}}

\numberwithin{equation}{section}

\title[Asymptotic stability in the $\phi^4$ model]{Asymptotic stability for odd perturbations of the stationary kink in the variable-speed $\phi^4$ model}
\author{Stanley Snelson}
\address{Department of Mathematics, University of Chicago, 5734 S. University Ave., Chicago, IL 60637}
\email{snelson@math.uchicago.edu}

\thanks{The author was partially supported by NSF grant DMS-1246999. The author would also like to thank Wilhelm Schlag for pointing out this problem and for helpful discussions.}

\begin{document}

\begin{abstract}We consider the $\phi^4$ model in one space dimension with propagation speeds that are small deviations from a constant function. In the constant-speed case, a stationary solution called the kink is known explicitly, and the recent work of Kowalczyk, Martel, and Mu\~noz established the asymptotic stability of the kink with respect to odd perturbations in the natural energy space. We show that a stationary kink solution exists also for our class of non-constant propagation speeds, and extend the asymptotic stability result by taking a perturbative approach to the method of Kowalczyk, Martel, and Mu\~noz. This requires an understanding of the spectrum of the linearization around the variable-speed kink.\end{abstract}
\maketitle
\section{Introduction}

The $\phi^4$ model is a classical nonlinear equation that arises in quantum field theory, statistical mechanics, and other areas of physics. See, for instance, \cite{MS, V, VS, Ra, W, PS} for the physical background. We are interested in the case where the propagation speed $c$ is allowed to vary with position. The equation is given in one space dimension by
\begin{equation}\label{e:main}
\partial_t^2\phi - c^2(x)\partial_x^2\phi = \phi - \phi^3, \quad (t,x)\in \R\times \R,
\end{equation}
where $c(x)$ is a uniformly positive function. We will restrict our attention to even functions $c$ that are small deviations from the constant unit speed $c\equiv 1$. (See below for the precise assumption.) Note that the energy
\[E(\phi,\partial_t\phi) := \int \frac 1 {c^2}\left(\frac 1 2 (\partial_t \phi)^2 + \frac 1 2 c^2 (\partial_x\phi)^2 + \frac 1 4 (1-\phi^2)^2\right)\dd x\]
is formally conserved if $\phi$ solves \eqref{e:main},

In the case $c(x)\equiv 1$, a stationary solution to \eqref{e:main} is known explicitly:
\[H(x) := \tanh\left(\frac x {\sqrt 2}\right),\]
known as the \emph{kink}, connects the two minima of the potential $\frac 1 4 (1-\phi^2)^2$ and is the unique bounded, odd solution of $-H''=H-H^3$, up to multiplication by $-1$. The kink in the $\phi^4$ model is seen as a prototype for solitons that occur in more complicated field theories, see \cite{Ra}. Since the energy $E(H,0)$ of the kink is finite, perturbations of the form $(\phi,\partial_t\phi) = (H+\vp_1,\vp_2)$ with $(\vp_1,\vp_2)\in H^1\times L^2$ are referred to as perturbations in the energy space. Standard arguments show that \eqref{e:main} is locally well-posed for initial data of the form $(H+\vp_1^{in},\vp_2^{in})$ with $\vp^{in}=(\vp_1^{in},\vp_2^{in})\in H^1\times L^2$.
Regarding the long-time behavior, in the constant-speed case, the kink is orbitally stable with respect to small perturbations in the energy space, by a result of Henry, Perez, and Wreszinski \cite{HPW}. In other words, solutions starting close to the kink remain close for all time, up to Lorentzian invariance. In a recent paper \cite{KMM}, Kowalczyk, Martel and Mu\~noz showed that in the case of odd perturbations (which corresponds to fixing the position of the traveling wave), this can be improved to asymptotic stability. Their approach, described below, is elementary and avoids the use of dispersive estimates. It was partially based on the work of Martel and Merle on the generalized KdV equations \cite{MM1, MM2} and Merle and Raphael \cite{MR} on the mass-critical nonlinear Schr\"odinger equation, but was adapted to additional difficulties resulting from the exchange of energy between internal oscillations and radiation, and the different decay rates of the corresponding components of the solution. These difficulties were seen earlier in the context of general Klein-Gordon equations with potential by Soffer and Weinstein \cite{SW}, who conjectured that a similar mechanism was at work in the $\phi^4$ model. The assumption of odd perturbations has appeared in other work concerning the asymptotic stability of solitons (see, for example, \cite{KK, KK2}), and in the $\phi^4$ model, odd perturbations already give rise to the challenging issues related to energy exchange. However, the authors of \cite{KMM} conjecture that the kink is in fact asymptotically stable with respect to general perturbations in the energy space. 

In this paper, we extend the results of \cite{KMM} to \eqref{e:main} with a certain class of non-constant propagation speeds $c(x)$. Before we state our results, it is convenient to exchange the second-order coefficient in \eqref{e:main} for a small first-order term by making the change of variables $ y = \int_0^x [1/c(s)]\dd s$. Defining $\Phi(t,y) = \phi(t,x(y))$, we obtain the equation
% and $K(y) = \bar K(x(y))$
\begin{align}
\partial_t^2\Phi - \partial_y^2\Phi + b(y) \partial_y \Phi &= \Phi - \Phi^3,\label{e:psi}
%-\partial_y^2 K  + b(y) \partial_y K &= K - K^3,\label{e:K}
\end{align}
with $b(y) = \left(1/ c(x(y))\right)\frac d {dy} c(x(y))$. We will deal with drift coefficients $b$ that are odd and satisfy
\begin{equation}\label{e:decay}
|b(y)|  \lesssim \delta e^{-\sqrt 2 |y|},\qquad  |b'(y)| \lesssim \delta,
\end{equation}
for some small constant $\delta>0$. In terms of $x$, it is sufficient to assume in \eqref{e:main} that $c(x) = 1 + c_\delta(x)$, with $c_\delta$ even, twice differentiable, and
\begin{equation*} 
|c_\delta(x)| + |c_\delta'(x)| \leq \delta e^{-c_1|x|}, \qquad  |c_\delta''(x)| \leq \delta,
\end{equation*}
with $c_1 =\sqrt 2/(1-\delta)$. We will work in the $y$ variable for the entire paper. Note that oddness in $y$ is equivalent to oddness in $x$, and that solutions to \eqref{e:main} and \eqref{e:psi} are odd if the initial data are odd.

Our first goal is the existence of a stationary solution in the variable-speed case, which is close to the constant-speed kink $H$ in the appropriate sense:
\begin{theorem}\label{t:exist}
Assume that $b$ satisfies \eqref{e:decay}. Then there exists an odd, bounded, time-independent solution $K$ of \eqref{e:psi}. Furthermore, for $H(y)=\tanh(y/\sqrt 2)$, the difference $H_\delta := K-H$ satisfies $|H_\delta(y)| + |H_\delta'(y)| \lesssim \delta e^{-\sqrt 2 |y|}$.
\end{theorem}
See Section \ref{s:stat} for the proof. We refer to $K(y)$ as the stationary kink, by analogy with the constant-speed case.

To study the long-time asymptotics of odd perturbations of $K(y)$ in the energy space, let $\vp(t) = (\varphi_1(t), \varphi_2(t))\in H^1\times L^2$ be odd in $y$, and set $\Phi = K + \varphi_1, \partial_t\Phi = \varphi_2$ in \eqref{e:psi}. Then the perturbation $\vp$ satisfies
\begin{equation}\label{e:perturb}
\begin{cases}\partial_t\vp_1 &= \vp_2\\
\partial_t\vp_2 &= -\mathcal L_K \vp_1 - (3K\vp_1^2+\vp_1^3),\end{cases}
\end{equation}
where $\mathcal L_K$ is the linearized operator around $K$:
\begin{equation}\label{e:LK}
\mathcal L_K = -\partial_y^2 - b(y)\partial_y - 1 + 3K^2 = \mathcal L - b(y)\partial_y + d(y).
\end{equation}
Here $\mathcal L = -\partial_y^2 - 1 + 3H^2$ is the linearization around $H(y) = \tanh(y/\sqrt 2)$, and $d(y) = 3(K(y)^2 - H(y)^2)$. With the inner products
\[\langle f,g\rangle := \int_\R f(y)g(y)\dd y, \qquad \langle f,g\rangle_p := \int_\R p(y)f(y)g(y)\dd y,\]
where $p(y) = \exp(\int_0^y b(s)\dd s)$, note that $\mathcal L$ is self-adjoint with respect to $\langle \cdot,\cdot\rangle$, and $\mathcal L_K$ is self-adjoint with respect to $\langle \cdot, \cdot\rangle_p$.

We now state our main theorem, which says that $K(y)$ is asymptotically stable with respect to odd perturbations in the energy space:
%3H_\delta^2(y)+6H(y)H_\delta(y)$.
%If we let
%\[\phi = \bar K + \bar \vp_1, \quad \partial_t \phi = \bar \vp_2,\]
%in \eqref{e:main}, then $(\bar \vp_1,\bar \vp_2)$ satisfies
%\begin{equation}\label{e:perturb1}
%\begin{cases}\partial_t\bar \vp_1 &= \bar \vp_2\\
%\partial_t\bar \vp_2 &= c^2(x)\partial_x^2\bar \vp_1 + (3\bar K^2 - 1)\bar \vp_1 -(3\bar K\bar \vp_1^2+\bar \vp_1^3),\end{cases}
%\end{equation}
%In this regard, the main theorem of this paper is the following:
\begin{theorem}\label{t:main}
There exist $\delta>0,\eps_0>0$ such that for any $\eps\in (0,\eps_0)$ and for any odd $\vp^{in}\in H^1\times L^2$ with 
\[\|\vp^{in}\|_{H^1\times L^2} < \eps,\]
the solution $\vp$ of \eqref{e:perturb} with $b$ satisfying \eqref{e:decay} and with initial data $\vp(0) = \vp^{in}$ exists globally in $H^1\times L^2$ and satisfies 
\[\lim_{t\rightarrow\pm \infty} \| \vp(t)\|_{H^1(I)\times L^2(I)}=0,\]
for any bounded interval $I\subset\R$.
\end{theorem}
The conclusion of Theorem \ref{t:main} cannot be improved to $\lim_{t\to\infty} \|\vp(t)\|_{H^1(\R)\times L^2(\R)} = 0$ because, by orbital stability and energy conservation of $\vp$ (see Section \ref{s:orbit}), this would imply $\vp(t) \equiv 0 $ for all $t\in \R$.

We now briefly describe the proof in \cite{KMM} of asymptotic stability in the constant-speed case. A key idea in that proof was to decompose the solution $\vp(t)$ based on the spectrum of the linearized operator $\mathcal L$. It is known (see, for example, \cite{NU}) that the spectrum
\[\sigma(\mathcal L) = \left\{0,\frac 3 2\right\} \cup [2,\infty),\]
with simple eigenvalues $0$ and $\frac 3 2$ corresponding to the $L^2$-normalized eigenfunctions
\[Y_0(x) := \frac 1 2 \sech^2\left(\frac x{\sqrt 2}\right)\]
and
\[Y_1(x) := 2^{-3/4}3^{1/2}\tanh\left(\frac x{\sqrt 2}\right)\sech\left(\frac x{\sqrt 2}\right).\]
Since $Y_0$ is even, it does not influence the dynamics of odd perturbations. But the odd eigenfunction $Y_1$, known as the internal mode of oscillation, plays a crucial role in the analysis.
The solution $\vp$ in the case $b\equiv d\equiv 0$ is written $\vp_1 = z_1(t)Y_1 + u_1$, $\vp_2 = (\frac 3 2)^{1/2}z_2(t)Y_1 + u_2$, with $\langle u_1,Y_1\rangle = \langle u_2,Y_1\rangle =0$, and asymptotic stability is proven as a consequence of an estimate of the form
\begin{equation}\label{e:z4}
\int_\R (|z_1(t)|^4 + |z_2(t)|^4)\dd t + \int_\R\int_\R ((\partial_x u_1)^2 + u_1^2 + u_2^2)e^{-c_0 |x|} \dd x \dd t \lesssim \|\vp^{in}\|_{H^1\times L^2}^2,\end{equation}
for some $c_0>0$. This estimate suggests that the internal oscillation mode $z(t) = (z_1,z_2)$ decays at a slower rate as $t\to\infty$ than $u =(u_1,u_2)$, which corresponds to radiation. After defining $v_1,v_2,\alpha,$ and $\beta$ in terms of $u$ and $z$ (in a formally similar way to the analogous quantities defined in Section \ref{s:orbit} below), the authors of \cite{KMM} proved \eqref{e:z4} using Virial functionals of the form $\mathcal I = \int \psi(\partial_x v_1)v_2 + \frac 1 2\int \psi'v_1v_2$ and $\mathcal J = \alpha \int v_2 g - 2\sqrt{3/2}\, \beta \int v_1 g$, with the functions $\psi$ and $g$ chosen advantageously. Using orbital stability and the equations for $(v_1,v_2)$ and $(\alpha,\beta)$ (which come from the equations for $(u_1,u_2)$ and $(z_1,z_2)$), it was found that
\begin{equation}\label{e:-ddt}
-\frac d {dt}(\mathcal I + \mathcal J) = \mathcal B(v_1) + \alpha\langle v_1,\tilde h\rangle + \alpha^2\langle f,g\rangle + \eps \,\mathcal O\left(|z|^4, \|v\|_{H_\omega^1\times L_\omega^2}^2\right),
\end{equation}
where $\mathcal B$ is a quadratic form, $f$ and $\tilde h$ are given Schwartz functions, and $H_\omega^1\times L_\omega^2$ is an exponentially weighted Sobolev space, see \eqref{e:weighted} for the definition. Next, the following coercivity result was established:
\begin{equation}\label{e:Bv1}
 \mathcal B(v_1) + \alpha\langle v_1,\tilde h\rangle + \alpha^2\langle f,g\rangle \gtrsim \alpha^2 + \|v_1\|_{H_\omega^1}^2,
 \end{equation}
for all odd $v_1\in H_\omega^1$ satisfying $\langle v_1,Y_1\rangle = 0$. Since, roughly speaking, $\alpha^2 \sim |z|^4$ and $v_i\sim u_i+ |z|^2$, this demonstrates that $\mathcal I$ and $\mathcal J$ are well adapted to the different decay rates of $z$ and $u$ that appear in \eqref{e:z4}. The proof of \eqref{e:Bv1} relied on delicate explicit estimates and changes of variables. The choice of the function $g$ was related to a nonlinear version of the Fermi-Golden rule (see \cite{RS, Si}), a non-resonance condition that ensures the internal oscillations are coupled to radiation, so that the energy of the system eventually radiates away from the kink; see \cite{KMM} for the details, and also \cite{SW, Sig} for the use of the same non-resonance condition in different contexts. The coercivity result \eqref{e:Bv1} and other estimates on $\alpha,\beta,v_1,$ and $v_2$, combined with the orbital stability of $H$, were used to establish \eqref{e:z4}.

To apply a similar method to \eqref{e:perturb} in the variable-speed case, where $b(y)$ and $d(y)$ in \eqref{e:LK} are nonzero, it is first of all necessary to understand how the spectrum of $\mathcal L_K$ differs from the spectrum of $\mathcal L$. In Section \ref{s:spec} below, we use ODE techniques to show that $\mathcal L_K$ has two simple eigenvalues $\lambda_0$ and $\lambda_1$ that are $\delta$-close to $0$ and $\frac 3 2$, and which correspond to an even eigenfunction $\bar Y_0$ and an odd eigenfunction $\bar Y_1$ respectively, which are exponentially decaying and close to $Y_0$ and $Y_1$ in $L^\infty$. With this information, in Section \ref{s:orbit} we establish the orbital stability of $K$ with respect to odd perturbations (Proposition \ref{p:orbit}) following the argument outlined in \cite{KMM}, and we perform a spectral decomposition of $\vp$ that is formally the same as in the constant-speed case. Namely, we write $\vp_1 = z_1(t)\bar Y_1 + u_1$, $\vp_2 = z_2(t)\bar Y_1 + u_2$ with $\langle u_1,\bar Y_1\rangle_p = \langle u_2, \bar Y_1\rangle_p = 0$, and define $\alpha,\beta,v_1,$ and $v_2$ in terms of $z(t)$ and $u(t)$. In Section \ref{s:vir}, we study the system for $(v_1,v_2,\alpha,\beta)$ with the same Virial functionals $\mathcal I$ and $\mathcal J$ mentioned above, with $\mu = \sqrt{\lambda_1}$ replacing $\sqrt{3/2}$ and a modified function $\bar g$ replacing $g$ in the definition of $\mathcal J$ (see Lemma \ref{l:barh} for the choice of $\bar g$). We find an expression for $\frac d {dt}(\mathcal I+\mathcal J)$ that is morally similar to \eqref{e:-ddt}. Since $\|Y_1 - \bar Y_1\|_{L^\infty} \lesssim \delta$ (Theorem \ref{t:spectrum}) and $\|p-1\|_{L^\infty}\lesssim \delta$, our $v_1$ satisfying $\langle v_1,\bar Y_1\rangle_p = 0$ will satisfy $\langle v_1, Y_1\rangle=0$ up to a small error which can be controlled in terms of $\|v_1\|_{H_\omega^1}$. This allows us to derive our coercivity result (Lemma \ref{l:D}) as a consequence of \eqref{e:Bv1} and perturbation arguments. This uses heavily the smallness assumption \eqref{e:decay} for $b$; to apply this type of method in the case where the propagation speed $c(x)$ may have large deviations from $c=1$, Virial functionals that are more specifically adapted to the resulting linear equation would likely be needed. After deriving Lemma \ref{e:D}, the conclusion of the argument (Section \ref{s:conclusion}) mainly involves controlling the higher-order terms in the dynamics of $\alpha,\beta,v_1$, and $v_2$, in much the same way as in \cite{KMM}.

Let us mention the following related results: Cuccagna \cite{C1} showed that the one-dimensional kink $H$, considered as a planar wave front in the constant-speed $\phi^4$ model in $\R^3$, is asymptotically stable with respect to general (not necessary odd) compactly supported, three-dimensional perturbations. This proof makes use of dispersive estimates due to Weder \cite{W1, W2} and relies on the better decay of these estimates available in three dimensions than in one (see also \cite{GS}). Other field equations that admit stationary kinks include the sine-Gordon equation $\partial_t^2u-\partial_x^2 u + \sin u = 0$, which also admits a one-parameter family of odd, time-periodic solutions referred to as wobbling kinks (see \cite{CQS}). Because of these solutions, the stationary kink in the sine-Gordon equation is not asymptotically stable in the energy space. As in the constant-speed case, our Theorem \ref{t:main} rules out the existence of wobbling kinks in the $\phi^4$ model in a neighborhood of $K(y)$. The question of existence or non-existence of wobbling kinks in the $\phi^4$ model has attracted attention in the past, at least in the constant-speed case (see \cite{S, KS}). We also mention the relativistic Ginzburg-Landau equation given by $\partial_t^2 u - \partial_x^2 u = W'(u)$ where $W$ is a double-well potential. Under an assumption on $W$ that excludes the $\phi^4$ model, but guarantees the existence of a kink, Kopylova and Komech \cite{KK} established asymptotic stability of the kink with respect to odd perturbations, using an approach inspired by the work of Buslaev and Sulem \cite{BS} on soliton stability for nonlinear Schr\"odinger equations (see also \cite{BP1, BP2, C}). To the author's knowledge, there are no previous results in the literature dealing with the asymptotic stability of solitary waves in an equation with non-constant speed of propagation.

%The paper is organized as follows. In Section \ref{s:stat}, we demonstrate the existence of a perturbed kink (Theorem \ref{t:exist}) using ODE perturbation techniques. Section \ref{s:spec} is concerned with characterizing the spectrum of $\mathcal L_K$, 

\section{Existence of stationary solution}\label{s:stat}

The purpose of this section is to prove Theorem \ref{t:exist}. In that proof and throughout the paper, we will need to solve integral equations of Fredholm type on the positive real line. For this, we use the following standard lemma, which we prove for the convenience of the reader:
\begin{lemma}\label{l:Fredholm}
Let $g\in L^\infty([0,\infty))$. If
\[\nu : =  \sup_{0\leq y<\infty} \int_{0}^\infty |G(y,w)| \dd w< 1,\]
then there exists a unique solution to 
\[f(y) = g(y) + \int_0^\infty G(y,w) f(w)\dd w\]
given by
\begin{equation}\label{e:series}
f(y) = g(y) + \sum_{n=1}^\infty \int_{0}^\infty \cdots \int_{0}^\infty \prod_{i=1}^n G(y_{i-1},y_i) g(y_n) \dd y_n \cdots \dd y_1,
\end{equation}
with $y_0=y$. Furthermore, one has
\[\|f\|_{L^\infty([0,\infty))}\leq  \frac 1 {1-\nu}\|g\|_{L^\infty([0,\infty))}.\]
\end{lemma}

\begin{proof}
We check directly that the iteration \eqref{e:series} converges:
\begin{align*}
\left| \int_0^\infty\right.& \cdots \left.\int_0^\infty \prod_{i=1}^n G(x_{i-1},x_i) g(x_n)\dd x_n \cdots \dd x_1\right|\\
&\leq \|g\|_{L^\infty} \int_0^\infty \cdots \int_0^\infty \prod_{i=1}^{n-1} |G(x_{i-1},x_i)| \int_0^\infty |G(x_{n-1},x_n)| \dd x_n \cdots \dd x_1\\
&\leq  \|g\|_{L^\infty} \nu \int_0^\infty \cdots \int_0^\infty \prod_{i=1}^{n-1} |G(x_{i-1},x_i)| \dd x_{n-1} \cdots \dd x_1\\
&\leq \cdots\\
&\leq \|g\|_{L^\infty} \nu^n,
\end{align*}
so the series converges, and $\|f\|_{L^\infty} \leq \dfrac{1}{1-\nu} \|g\|_{L^\infty}$.
\end{proof}

Now we find a stationary solution to \eqref{e:psi}, i.e. an odd $K$ solving
\begin{equation}\label{e:K}
-\partial_y^2 K  + b(y) \partial_y K = K - K^3.
\end{equation}
\begin{proof}[Proof of Theorem \textup{\ref{t:exist}}]
We look for $H_\delta(y)$ such that $K(y) = H(y) + H_\delta(y)$ solves \eqref{e:K}, where $H(y) = \tanh(y/\sqrt 2)$ satisfies $-H_{yy} = H-H^3$. If $H_\delta(y)$ solves
\begin{equation}\label{e:Hdelta}
\begin{cases}-\partial_y^2H_\delta + b(y)\partial_y H_\delta + (3H^2-1)H_\delta = -b(y)\partial_yH - H_\delta^3 - 3HH_\delta^2,\\
H_\delta(0) = 0, \quad H_\delta \rightarrow 0 \mbox{ as } y\rightarrow \infty,\end{cases}
\end{equation}
then we can then extend $H_\delta$ to the real line by oddness and obtain $K=H+H_\delta$. We write \eqref{e:Hdelta} as
\begin{equation*}
\mathcal L_b H_\delta = -H_\delta^3 - 3HH_\delta -b(y)\partial_y H,
\end{equation*}
where 
\[\mathcal L_b = -\partial_y^2 + b(y) \partial_y+  (3H^2-1) = \mathcal L + b(y) \partial_y.\]
We will find $H_\delta$ by computing a Green's function for $\mathcal L_b$ on $[0,\infty)$. A fundamental system for $\mathcal L Y=0$ is given by 
\begin{align*}
Y_0(y) &= \frac 1 2 \sech^2(y/\sqrt 2),\\
Z_0(y) &= \int_0^y\cosh^4(s/\sqrt 2)\dd s\\
& =- \frac 1 {32} \sech^2(y/\sqrt 2)\left(12y + 8\sqrt 2\sinh(\sqrt 2 y) + \sqrt 2 \sinh(2\sqrt 2 y)\right).
\end{align*} 
To find $Y_b, Z_b$ with $\mathcal L_bY_b = \mathcal L_b Z_b = 0$, we first make the substitution $Y_b = Y_0 + V_b$, which leads to the equation
\[\mathcal L V_b = -b(y)\partial_y(Y_0+V_b)\]
for $V_b(y)$. This can be written as the integral equation
\begin{equation}\label{e:V}
V_b(y) = g(y) +\int_0^\infty G_0(y,w) b(w) \partial_wV_b(w)\dd w, \qquad y\geq 0,
\end{equation}
where
\[g(y) = \int_0^\infty G_0(y,w) b(w) \partial_{w} Y_0(w)\dd w,\]
and
\begin{equation}\label{e:G0}
G_0(y,w) = \begin{cases}Y_0(y)Z_0(w), &\quad 0\leq w < y,\\
						 Y_0(w)Z_0(y), &\quad 0\leq y <w,\end{cases}
\end{equation}
%With this Green's function, $V_0(0) = 0$ and $V_1(0)=0$ will hold.
Using $|Y_0(y)|\lesssim e^{-\sqrt 2|y|}, |Z_0(y)|\lesssim e^{\sqrt 2|y|}, |Y_0'(y)|\lesssim e^{-\sqrt 2 |y|}$, $|Z_0'(y)| \lesssim e^{2\sqrt 2|y|}$, and the bound \eqref{e:decay} for $b$, we see that
\begin{align}\label{e:f1}
g(y) & = Y_0(y)\int_0^y Z_0(w) b(w) \partial_w Y_0(w)\dd w + Z_0(y)\int_y^\infty Y_0(w)b(w)\partial_w Y_0(w)\dd w\nonumber\\
&\lesssim \delta\left(e^{-\sqrt 2 y} \int_0^y e^{-\sqrt 2 w} \dd w + e^{\sqrt 2 y}\int_y^\infty e^{-3\sqrt 2 w} \dd w \right)\\
&\lesssim \delta e^{-\sqrt 2 y}.\nonumber
\end{align}
%\begin{align}\label{e:f1}
%f_1(y) & = U_0(y)\int_0^y U_1(w) b(w) \partial_w U_1(w)\dd w + U_1(y)\int_y^\infty U_0(w)b(w)\partial_w U_1(w)\dd w\nonumber\\
%&\lesssim \delta\left(e^{-\sqrt 2 y} \int_0^y e^{(3\sqrt 2-c_0)w} \dd w + e^{\sqrt 2 y}\int_y^\infty e^{(\sqrt 2-c_0) w} \dd w \right)\\
%&\lesssim \delta e^{-\sqrt 2 y}.\nonumber
%\end{align}
%By similar reasoning, we have $|f_0(y)| \lesssim \delta e^{-\sqrt 2 y}$ also. 
Now, we integrate by parts in \eqref{e:V} to obtain the Fredholm equation
\[V_b(y) = g(y) - \int_0^\infty \partial_w\left[G_0(y,w)b(w)\right]V_b(w)\dd w.\]
There are no boundary terms because $b(0)=0$.
%, $G_0$ is continuous at $w=y$, and $b$ decays sufficiently fast as $y\rightarrow \infty$. (The last fact will be justified retroactively when we find $V_0$.) 
By \eqref{e:decay} and the above bounds on $Y_0$ and $Z_0$, we have
\begin{align*}
\sup_{[0,\infty)} \int_0^\infty |\partial_w[G_0(y,w)b(w)]|\dd w &\leq \sup_{[0,\infty)}\left( |Y_0(y)|\int_0^y|Z_0'(w)b(w) + Z_0(w)b'(w)|\dd w\right.\\
&\qquad\qquad\left.+|Z_0(y)|\int_y^\infty|Y_0'(w)b(w)+Y_0(w)b'(w)|\dd w\right)\\
&\leq \sup_{[0,\infty)} C\delta \left(e^{-\sqrt 2 y}(e^{\sqrt 2y}+e^{\sqrt 2 y})\right.\\
&\quad \qquad\left. + e^{\sqrt 2 y}(e^{-2\sqrt 2y} + e^{-\sqrt 2 y}) \right)<1,
\end{align*}
if $\delta$ is sufficiently small, so by Lemma \ref{l:Fredholm}, a unique solution $V_b$ exists, and $\|V_b\|_{L^\infty} \leq C\|g\|_{L^\infty}\leq C\delta$.  It is clear from formula \eqref{e:series} in Lemma \ref{l:Fredholm} and the decay of $g$ that $|V_b| = |Y_b-Y_0| \lesssim \delta e^{-\sqrt 2y}$.

Using reduction of order, we obtain a second independent solution $Z_b$ given by
\[Z_b(y) = Y_b(y)\int_0^y \frac{\exp(\int_0^w b(s)\dd s)}{(Y_b(w))^2}\dd w.\]
We have $Z_b(0) = 0$, $Z_b'(0) = 1$, and $Z_b(y) \lesssim e^{\sqrt 2 y}$. Let $p = Y_b Z_b' - Y_b'Z_b   = \exp(\int_0^y b(s)\dd s)$, and define the Green's function $G_b$ for $\mathcal L_b$:
\[G_b(y,w) = \begin{cases}Y_b(y)Z_b(w)/p(w), &\quad 0\leq w < y,\\
						  Y_b(w)Z_b(y)/p(w), &\quad 0\leq y <w.\end{cases}\]
Note that $G_b(0,w) = 0$. 

We can now write \eqref{e:Hdelta} as a nonlinear integral equation for $H_\delta(y)$:
\begin{equation}\label{e:TH}
H_\delta(y) =  (\mathcal T H_\delta)(y) := h(y) - \int_0^\infty G_b(y,w)\left[H_\delta^3(w) + 3H(w)H_\delta^2(w)\right]\dd w,
\end{equation}
where $h(y) = -\int_0^\infty G_b(y,w)b(w)\partial_w H(w)\dd w$. We will show that $\mathcal T$ has a unique fixed point in a suitable class.  Define the norm
\[\|\eta\|_{\sim} := \sup_{0\leq y<\infty} e^{\sqrt 2y}|\eta(y)|.\]
Note first that, since $\partial_y H = \frac 1 {\sqrt 2} \sech^2(y/\sqrt 2)$, we have $|b(w) H'(w)| \lesssim  \delta  e^{-2\sqrt 2|w|}$. By estimating the integral in a similar manner to \eqref{e:f1}, we see that $|h(y)|\leq C_1 \delta e^{-\sqrt2 y}$ for some constant $C_1$.
Let $C_0 = 2C_1$, and define the set $\mathcal A_\delta = \{\eta \in C([0,\infty)): \|\eta(y)\|_{\sim}\leq C_0\delta \}$. For $\eta \in \mathcal A_\delta$, we check directly that 
\[|\eta^3(w) + 3H(w)\eta^2| \leq 4\delta^2 C_0^2 e^{-2\sqrt 2w},\]
and
\begin{align*}
\Big|\int_0^\infty &G_b(y,w)\left[\eta^3(w) + 3H(w)\eta^2(w)\right]\dd w\Big|\\ 
&\leq 4\delta^2 C_0^2 \left(|Y_b(y)|\int_0^y \frac{|Z_b(w)|}{p(w)} e^{-2\sqrt 2 w}\dd w + |Z_b(y)|\int_y^\infty \frac{|Y_b(w)|}{p(w)} e^{-2\sqrt 2 w}\dd w\right)\\
&\leq 4\delta^2 C_0^2 C_2e^{-\sqrt 2y}.
\end{align*}
Then, if $\delta <\dfrac 1 {8 C_0C_2}$, we have
\[|\mathcal T \eta (y)| \leq C_1\delta e^{-\sqrt 2 |y|} + 4\delta^2 C_0^2 C_2 e^{-\sqrt 2|y|} < 2C_1 \delta e^{-\sqrt 2|y|}, \]
so $\mathcal T\eta \in \mathcal A_\delta$. Finally, for $\eta_1, \eta_2\in \mathcal A_\delta$, we have
\begin{align*}
\left|\eta_1^3 - \eta_2^3 + 3H(\eta_1^2-\eta_2^2)\right| &\leq |\eta_1 - \eta_2|\left|\eta_1^2 + \eta_1\eta_2+ \eta_2^2 + 3H(\eta_1+\eta_2)\right|\\ 
&\leq 9C_0\delta e^{-\sqrt 2 y}|\eta_1 - \eta_2|,
\end{align*}
and proceeding as before,
\begin{align*}
|\mathcal T(\eta_1)(y) - \mathcal T(\eta_2)(y)| &\leq 9C_0\left(\delta |Y_b(y)|\int_0^y \frac{|Z_b(w)|}{p(w)}e^{-\sqrt 2 w}|\eta_1(w)-\eta_2(w)|\dd w\right.\\
&\qquad\qquad\left.+ |Z_b(y)|\int_y^\infty \frac{|Y_b(w)|}{p(w)}e^{-\sqrt 2 w}|\eta_1(w)-\eta_2(w)|\dd w\right)\\
&\lesssim \delta e^{-\sqrt 2y}\|\eta_1 - \eta_2\|_{\sim},
\end{align*}
so that $\|\mathcal T(\eta_1) - \mathcal T(\eta_2)\|_{\sim} \leq C\delta \|\eta_1 - \eta_2\|_{\sim}.$ If $\delta < 1/C$, then $\mathcal T$ is a contraction in $\mathcal A_\delta$ and a unique solution $H_\delta$ to \eqref{e:Hdelta} exists in $\mathcal A_\delta$.

By differentiating \eqref{e:TH}, we verify that 
\[|H_\delta(y)| + |H_\delta'(y)| \lesssim \delta e^{-\sqrt 2|y|}.\]
\end{proof}

\section{Spectrum of linearized operator}\label{s:spec}

We now analyze the spectrum of $\mathcal L_K = -\partial_y^2 -b(y)\partial_y -1 + 3K^2$. This operator can be written $\mathcal L_K = \mathcal L - b(y)\partial_y + d(y)$, a perturbation of the classical operator $\mathcal L = -\partial_y^2 -1 + 3H^2(y)$. Here, $d(y) = 3H_\delta^2(y) + 6H(y)H_\delta(y)$. We find that the $L^2$-spectrum $\sigma(\mathcal L_K)$ of $\mathcal L_K$ is qualitatively similar to the spectrum of $\mathcal L$ in the following sense:
\begin{theorem}\label{t:spectrum}
The operator $\mathcal L_K$ has real, simple eigenvalues $\lambda_0$, $\lambda_1$ such that $|\lambda_0|\lesssim \delta$ and $|\lambda_1 - \frac 3 2|\lesssim \delta$. 
%The spectrum $\sigma(\mathcal L_K) = \{\lambda_0,\lambda_1\}\cup[2,\infty)$, where $\lambda_i$ are real, simple eigenvalues, $|\lambda_0|\lesssim \delta$, and $|\lambda_1 - \frac 3 2|\lesssim \delta$. 
The corresponding eigenfunctions $\bar Y_0$ and $\bar Y_1$ are even and odd respectively  %exponentially decaying at the rates $|\bar Y_0(y)|\lesssim e^{-\sqrt 2|y|}$, $|\bar Y_1(y)|\lesssim e^{-|y|/\sqrt 2}$, 
and satisfy 
\begin{align*}
|\bar Y_0(y) - Y_0(y)|+|\bar Y_0'(y) - Y_0'(y)| &\lesssim \delta e^{-\sqrt 2 |y|},\\
|\bar Y_1(y) - Y_1(y)|+|\bar Y_1'(y) - Y_1'(y)| &\lesssim \delta e^{- |y|/\sqrt 2},
\end{align*}
where $Y_0$ and $Y_1$ are the eigenfunctions of $\mathcal L$ corresponding to $0$ and $\frac 3 2$. Furthermore, $\lambda_1$ is the only discrete eigenvalue of $\mathcal L_K$ corresponding to an odd eigenfunction, and the continuous spectrum $\sigma_c(\mathcal L_K) = [2,\infty)$.
\end{theorem}
\begin{proof}
First, recall that $\mathcal L_K$ is self-adjoint with respect to the $\langle \cdot,\cdot\rangle_p$ inner product, so $\sigma(\mathcal L_K)\subset\R$. Next, by general theory (see, for example, \cite[Chapter 18]{HS}) the continuous spectrum of $\mathcal L$ is stable under the relatively compact perturbation $-b\partial_y + d$. (In other words, $(-b\partial_y+d)(\mathcal L - z)^{-1}$ is a compact operator for any $z\in \rho(\mathcal L)$.) Therefore, $\sigma_c(\mathcal L_K) = \sigma_c(\mathcal L) = [2,\infty)$. 

We now show that $\sigma(\mathcal L_K)$ lies inside the $C_0\delta$ neighborhood of $\sigma(\mathcal L)$ for some constant $C_0$. Assume that $\lambda \in \rho(\mathcal L)\cap \sigma(\mathcal L_K)$, where $\rho(\mathcal L)$ denotes the resolvent set of $\mathcal L$, and let $d_0 = \dst(\lambda,\sigma(\mathcal L))$. We may assume $|\lambda|\leq 3$ because elliptic existence theory implies $(-\infty,-3) \subset\rho(\mathcal L_K)$. Since $\lambda\in\rho(\mathcal L)$, for $w\in L^2(\R)$ we have
\[\|(\mathcal L - \lambda I)^{-1} w\| \leq \frac{\|w\|}{d_0}.\]
(For the duration of this proof, $\|\cdot\|$ denotes the norm in $L^2(\R)$.) This is equivalent to $\|(\mathcal L-\lambda I)v\| \geq d_0\|v\|$ for all $v\in D(\mathcal L)$. Since $\lambda \in \sigma(\mathcal L_K)$, there exists a sequence $v_n\in D(\mathcal L_K)= D(\mathcal L)$ such that $\|v_n\|=1$ and $(\mathcal L_K - \lambda I)v_n\rightarrow 0$ in $L^2(\R)$. But since $\|(\mathcal L - \lambda I)v_n\| \geq d_0$, we have 
\[\|(\mathcal L_K - \mathcal L)v_n\| = \|bv_n' + dv_n\| \geq \frac {d_0} 2\] 
for $n$ sufficiently large. It is clear that $\|bv_n'\| \lesssim \delta \|v_n'\|$. Looking at $\|v_n'\|$, we have
\begin{align*}
\int (v_n')^2 = \int v_n(-v_n'') &= \int[ v_n((\mathcal L_K - \lambda I)v_n - bv_n' - (3H^2-1+d-\lambda)v_n)]\\
&\leq \|v_n\|\|(\mathcal L_K - \lambda I)v_n\| + \frac 1 2 \int b'(v_n)^2 + C\|v_n\|\\
&\leq \|(\mathcal L_K - \lambda I)v_n\| + C\|v_n\|.
\end{align*}
Since $\|(\mathcal L_K - \lambda I)v_n\|\rightarrow 0$ by assumption, we have that for $n$ sufficiently large,
\[\frac {d_0}2 \leq \|bv_n'\| + \|d v_n\|\lesssim \delta (\|v_n'\| + \|v_n\|) \lesssim \delta,\]
or $\dst(\lambda,\sigma(\mathcal L))\leq C_0\delta$, as desired.

Next, we show that $\mathcal L_K$ has exactly one eigenvalue in $[-C_0\delta,C_0\delta]$. For some $\lambda_\ast$ to be determined satisfying $\lambda_\ast\geq C_0\delta$ but $|\lambda^*|\lesssim \delta$,  we take $\lambda\in[-\lambda_\ast,\lambda_\ast]$ and look for $\bar Y_0(y)\in L^2$ satisfying $\mathcal L_K \bar Y_0 = \lambda \bar Y_0$. Letting $\bar Y_0 = Y_0 + U_\lambda$, we obtain the following equation for $U_\lambda$:
\begin{equation}\label{e:eig0}
\mathcal L U_\lambda =   bY_0'+(\lambda -d)Y_0 + bU_\lambda' + (\lambda -d)U_\lambda.
\end{equation}
Note that the solution to this equation on $(-\infty,\infty)$ must be even, because otherwise, writing $U_\lambda = U^o + U^e$, the odd part would satisfy $\mathcal L U^o = b(U^o)' + (\lambda - d)U^o$, which implies $\langle U^o,\mathcal L U^o\rangle = \langle U^o, b(U^o)' + (\lambda-d)U^o\rangle \lesssim \delta \|U^o\|^2$, a contradiction because $U^o$ is orthogonal to the even eigenfunction $Y_0$, so by the spectral theorem, $\langle U^o, \mathcal L U^o\rangle \geq \frac 3 2 \|U^o\|^2$.

We write \eqref{e:eig0} on $[0,\infty)$ as the integral equation
\begin{equation}\label{e:Z0}
U_\lambda(y) = h_0(y) + \int_0^\infty G_0(y,w) [b(w)U_\lambda'(w) + (\lambda - d(w))U_\lambda(w)]\dd w,
\end{equation}
where $G_0$ is the Green's function defined in \eqref{e:G0}, and
\begin{align*}
h_0(y) &= \int_0^\infty G_0(y,w) [b(w)Y_0'(w) + (\lambda - d(w))Y_0(w)]\dd w\\
&=Y_0(y)\int_0^y Z_0(w) [b(w)Y_0'(w) + (\lambda - d(w))Y_0(w)]\dd w \\
&\quad + Z_0(y)\int_y^\infty Y_0(w)[b(w)Y_0'(w) + (\lambda - d(w))Y_0(w)]\dd w.
\end{align*}
By the asymptotics of $Y_0$ and $Z_0$, we have $|h_0(y)|\lesssim \delta ye^{-\sqrt 2 y}$. To solve \eqref{e:Z0}, we check that
\begin{align*}
\int_{0}^\infty  |\partial_w G_0(y&,w)b(w) + (b'(w) - \lambda + d(w)) G_0(y,w)| \dd w\\
&\lesssim \delta \left(Y_0(y)\int_0^y (Z_0'(w)e^{-\sqrt 2 w}+Z_0(w))\dd w + Z_0(y)\int_y^\infty Y_0(w)\dd w \right)\\
&\lesssim \delta,
\end{align*}
uniformly in $y\geq 0$. (Recall that $Z_0'(w)\lesssim e^{2\sqrt 2 w}$.)
%)for all $y\geq 0$. which is satisfied as long as $|\lambda|\lesssim \delta$, because 
%\[\int_0^\infty |G_0(y,w)|\dd w \lesssim \left(e^{-\sqrt 2 y}\int_0^y e^{\sqrt 2 w} \dd w + e^{\sqrt 2 y}\int_y^\infty e^{-\sqrt 2 w}\dd w \right)\leq  C,\]
%where $C$ is independent of $y$, and $|b' + \lambda -d|\lesssim \delta$. (The $\partial_wG_0 b$ term will not cause any problems because of the decay of $b$.) 
Lemma \ref{l:Fredholm} implies $U_\lambda$ exists on $[0,\infty)$ for each $\lambda$, $\|U_\lambda\|_{L^\infty} \lesssim \|h_0\|_{L^\infty} \lesssim \delta$, and $|U_\lambda(y)|\lesssim \delta e^{-\sqrt 2 y}$.
To extend by evenness to the real line, we would need $U_\lambda'(0) = 0$. Note that since $Z_0'(0) = 1$, \eqref{e:Z0} implies
\begin{align}\label{e:Uprime}
U_\lambda'(0) &= \int_0^\infty Y_0[b(Y_0+U_\lambda)' + (\lambda - d)(Y_0+U_\lambda)]\dd w\nonumber\\ 
&= \lambda \int_0^\infty Y_0(Y_0+U_\lambda)\dd w - \int_0^\infty [(d + b')Y_0 + bY_0'](Y_0+U_\lambda) \dd w.
\end{align}
Since $\int_0^\infty Y_0(Y_0+U_\lambda) \geq \frac 1 2-C\delta\geq \frac 1 4$ and $\|Y_0+U_\lambda\|_{L^\infty}\leq 1$, we choose 
\[\lambda_\ast = \max\left(C_0\delta, \,5\int_0^\infty \left|(d+b')Y_0+bY_0'\right|\dd w\right),\]
so that $U_{\lambda_\ast}'(0) >0$, $U_{-\lambda_\ast}'(0)<0$, and $|\lambda_\ast|\lesssim \delta$. We will now show that $U_\lambda'(0)$ depends on $\lambda$ in a continuous and monotonic way.

For $\lambda, \mu \in [-\lambda_\ast,\lambda_\ast]$, observe that $\Delta = U_{\lambda} - U_{\mu}$ satisfies 
\[\Delta(y) = g_\Delta(y) + \int_0^\infty G_0(y,w)(b\Delta' - d\Delta)\dd w,\]
with
\[g_\Delta(y) = \int_0^\infty G_0(y,w)(\lambda U_{\lambda} - \mu U_{\mu})\dd w.\]
Since $|\lambda U_{\lambda} - \mu U_{\mu}|\lesssim \delta e^{-\sqrt 2 y}$, this can be solved as above, using Lemma \ref{l:Fredholm}, and $\|\Delta\|_{L^\infty} \lesssim \|g_\Delta\|_{L^\infty}$. We have
\[\|g_\Delta\|_{L^\infty} \lesssim \|\lambda U_{\lambda} - \mu U_{\mu}\|_{L^\infty} = \|(\lambda - \mu)U_{\lambda} + \mu \Delta\|_{L^\infty} \leq C_1|\lambda-\mu| + C_2\delta \|\Delta\|_{L^\infty}.\]
Combining this with $\|\Delta\|_{L^\infty} \lesssim \|g_\Delta\|_{L^\infty}$, we conclude $\|U_{\lambda} - U_{\mu}\|_{L^\infty} \lesssim |\lambda - \mu|$ if $\delta$ is sufficiently small. 

Let $\lambda > \mu$. By \eqref{e:Uprime}, we have
\begin{align}\label{e:mono}
U_{\lambda}'(0) -U_{\mu}'(0) &=  (\lambda - \mu)\int_0^\infty Y_0^2 \dd w\\
& + \int_0^\infty [ Y_0(\lambda U_{\lambda} - \mu U_{\mu}) - (U_{\lambda} - U_{\mu})(d+b')Y_0 + bY_0')]\dd w.\nonumber
\end{align}
Since $\|U_{\lambda} - U_{\mu}\|_{L^\infty} \lesssim |\lambda - \mu|$ and $\|\lambda U_{\lambda} - \mu U_{\mu}\|_{L^\infty} \lesssim \delta |\lambda - \mu|$, the second integral in \eqref{e:mono} is bounded in absolute value by a constant times $\delta |\lambda-\mu|$. This implies $U_{\lambda}'(0) > U_{\mu}'(0)$ and that $U_\lambda'(0)$ depends continuously on $\lambda$. We conclude $U_\lambda'(0) = 0$ for a unique $\lambda_0\in [-\lambda_\ast,\lambda_\ast]$. This $\lambda_0$ is an eigenvalue of $\mathcal L_K$ corresponding to the even, exponentially decaying eigenfunction $\bar Y_0 = Y_0 + U_{\lambda_0}$. Differentiating \eqref{e:Z0}, we conclude $|\bar Y_0'(y) - \bar Y_0(y)|\lesssim \delta e^{-\sqrt 2|y|}$.

Now we will find an eigenfunction $\bar Y_1$ corresponding to some $\lambda_1$ close to $\frac 3 2 $. For $\mathcal L$, note that 
\begin{align*}
Y_1(y) &= 2^{-3/4}3^{1/2} \tanh\left(\frac y{\sqrt 2}\right)\sech\left(\frac y{\sqrt 2}\right)\\
Z_1(y) &= -\dfrac 1{4} \sech\left(\frac y {\sqrt 2}\right)\left[-5 + 3\sqrt 2 y\tanh\left(\frac y {\sqrt 2}\right) + \cosh\left(\sqrt 2 y\right)\right]
\end{align*}
form a fundamental system for $\mathcal L - \frac 3 2 I$ on $[0,\infty)$ with $Y_1(0) = 0$ and $Z_1'(0) = 0$. Following the above method, we will take $\lambda\in [\frac 3 2 - \lambda^\ast,\frac 3 2 + \lambda^\ast]$ with $|\lambda^\ast|\lesssim \delta$ to be determined. If $\bar Y_1$ satisfies $\mathcal L_K \bar Y_1 =\lambda \bar Y_1$ on $[0,\infty)$, then letting $\bar Y_1 = Y_1 + V_\lambda$, we have
\begin{equation}\label{e:Vlambda}
\mathcal L V_\lambda -\frac 3 2 V_\lambda = bY_1' + \left(\lambda - \frac 3 2 - d\right)Y_1 + b V_\lambda ' + \left(\lambda - \frac 3 2 - d\right)V_\lambda.
\end{equation}
Similarly to above, we write $V_\lambda = V^e + V^o$. If the even part $V^e \not\equiv 0$, then $V^e$ satisfies $\mathcal L V^e = b(V^e)' + (\lambda-d)V^e$, so that
\begin{equation}\label{e:LVe}
\left|\frac{\langle V^e, \mathcal L V^e\rangle} {\|V^e\|^2}-\frac 3 2\right| \lesssim \delta.
\end{equation} 
However, for $V^e$ we have
 \[ 0 = \langle \bar Y_0, \bar Y_1\rangle_p = \langle \bar Y_0, Y_1 + V^e + V^o\rangle_p = \langle \bar Y_0, V^e\rangle_p,\]
which implies
\[|\langle Y_0, V^e \rangle_p| = |\langle \bar Y_0 - Y_0, V^e\rangle_p| \leq  \|Y_0-\bar Y_0\|_{L^\infty(\R)}\|p\|\|V^e\|\lesssim \delta \|V^e\|,\]
and therefore, 
\[|\langle Y_0, V^e\rangle| \leq |\langle Y_0, V^e\rangle_p| + \|p-1\|_{L^\infty(\R)} |\langle Y_0,V^e\rangle_p|,\]
so that $|\langle Y_0, V^e\rangle|\lesssim \delta \|V^e\|$. Now, since $\langle Y_1, V^e\rangle = 0$, we can write $V^e = a_0 Y_0 + a_1 W$, with $\langle Y_0, W\rangle = 0$ and $a_0 = \langle Y_0, V^e\rangle$. By the spectral theorem, $\langle W, \mathcal L W\rangle/\|W\| \geq 2$, which contradicts \eqref{e:LVe} because $\langle V^e, \mathcal L V^e\rangle = a_1^2\langle W, \mathcal L W\rangle$ and $\|W\| \geq (1-C\delta)\|V^e\|$ for some $C$. We conclude $V_\lambda$ is odd.
 
 On $[0,\infty)$, \eqref{e:Vlambda} is equivalent to
\begin{equation}\label{e:Vintegral}
V_\lambda = h_1(y) + \int_0^\infty G_1(y,w)[b(w)V_\lambda'(w) + (\lambda - \frac 3 2 - d(w))V_\lambda(w)]\dd w,
\end{equation}
where $h_1(y) = \int_0^\infty G_1(y,w)[b(w)Y_1'(w) + (\lambda - \frac 3 2 - d(w))Y_1(w)]\dd w$ and
\begin{equation*}
G_1(y,w) = \begin{cases} Y_1(y)Z_1(w), &0\leq w < y ,\\
				       Z_1(y)Y_1(w), &0\leq y < w.\end{cases}
\end{equation*}
Using the asymptotics $|Y_1(y)|+ |Y_1'(y)|\lesssim e^{-y/\sqrt 2}$ and $|Z_1(y)|+|Z_1'(y)|\lesssim e^{y/\sqrt 2}$, the same arguments used in solving \eqref{e:Z0} above imply that $|h_1(y)|\lesssim \delta y e^{-y/\sqrt 2}$, that 
\[\sup_{0<y<\infty} \left(\int_{0}^\infty  |\partial_w G_1(y,w)b(w) + (b'(w) + \lambda - \frac 3 2 - d(w)) G_1(y,w)| \dd w\right) < 1,\]
and therefore $V_\lambda$ solving \eqref{e:Vintegral} exists uniquely, and that $|V_\lambda(y)|+|V_\lambda'(y)|\lesssim \delta e^{-y/\sqrt 2}$.
%We verify that $h_1$ decays exponentially as $y\rightarrow\infty$:
%\begin{align*}
%h_1(y) &= Y_1(y)\int_0^y Z_1(w) [b(w)Y_1'(w) + (\lambda - \frac 3 2 - d(w))Y_1(w)]\dd w \\
%&\quad + Z_1(y)\int_y^\infty Y_1(w)[b(w)Y_1'(w) + (\lambda - \frac 3 2 - d(w))Y_1(w)]\dd w.
%\end{align*}
%As before, we have $|h_1(y)|\lesssim \delta y e^{-y/\sqrt 2}$. By Lemma \ref{l:Fredholm}, we can solve \eqref{e:Vlambda} because

%for $\delta$ sufficiently small, and $\|V_\lambda\|_{L^\infty} \lesssim \|h_1\|_{L^\infty} \lesssim \delta$. (We can check directly that $\int_0^\infty |G_1(y,w)|\dd w \leq C$ independently of $y$, as above.)

We need $V_\lambda(0)=0$ to extend $V_\lambda$ by oddness. Note that
\begin{align*}
V_\lambda(0) &= \int_0^\infty Y_1\left[b(Y_1'+V_\lambda')+\left(\lambda-\frac 3 2 - d\right)(Y_1 + V_\lambda)\right]\dd w\\
& = \left(\lambda-\frac 3 2\right)\int_0^\infty Y_1(Y_1+V_\lambda)\dd w - \int_0^\infty [(b'+d)Y_1 - bY_1'](Y_1+V_\lambda)\dd w
\end{align*}
since $Z_1(0) = 1$. We choose
\[\lambda^\ast  = \max\left( C_0\delta,  5\int_0^\infty \left|(b'+d)Y_1 - bY_1'\right|\dd w \right)\]
It is straightforward to check that $\|Y_1+V_\lambda\|_{L^\infty}\leq 1$, so this choice of $\lambda^\ast$ ensures $V_{3/2+\lambda^\ast}(0)>0$, $V_{3/2-\lambda^*}(0)<0$, and $|\lambda^\ast|\lesssim \delta$. Given $\lambda,\mu\in[\frac 3 2 - \lambda^\ast,\frac 3 2 + \lambda^\ast]$, we can show by arguments similar to above that $\|V_{\lambda} - V_{\mu}\|_{L^\infty} \lesssim |\lambda-\mu|$, that $|V_{\lambda}(0) - V_{\mu}(0)|\lesssim |\lambda - \mu|$, and that $V_{\lambda}(0) > V_{\mu}(0)$ if $\lambda >\mu$. We conclude there is a unique $\lambda_1\in [\frac 3 2-\lambda^*,\frac 3 2+\lambda^*]$ such that $V_{\lambda_1}(0) = 0$. Extending $V_{\lambda_1}$ by oddness, there is an odd, exponenentially decaying eigenfunction $\bar Y_1 = Y_1 + V_{\lambda_1}$ corresponding to $\lambda_1$.

For $\delta$ sufficiently small, the interval $[2-C_0\delta,2)$ contains at most one eigenvalue of $\mathcal L_K$. By general Sturm-Liouville theory, all eigenvalues of $\mathcal L_K$ are simple (indeed, one may compute directly that the Wronskian of two eigenfunctions is zero) and the parity of the eigenfunctions must alternate (because the eigenvalues of $\mathcal L_K$ on $[0,\infty)$ with Dirichlet and Neumann boundary conditions at $0$ must interlace). We conclude that $\bar Y_1$ is the only odd eigenfunction corresponding to the discrete spectrum of $\mathcal L_K$.
%The eigenvalues $\lambda_0$ and $\lambda_1$ are simple, as can be seen from the usual Sturm-Liouville proof: for two eigenfunctions $Y$ and $Z$ corresponding to $\lambda_i$, we compute directly that $W'(y) = \dfrac d{dy}(\exp(\int_0^y b(s)\dd s)(YZ'-Y'Z)) = 0$. Since $W(y)$ must approach $0$ as $y\rightarrow\infty$, it is identically zero, and we conclude $Y=\beta Z$ for some $\beta\in\R$.
\end{proof}

\section{Orbital stability and spectral decomposition}\label{s:orbit}

To prove the orbital stability with respect to odd perturbations $\vp$ solving \eqref{e:perturb}, we follow the outline of the simple proof given in \cite{KMM} for odd perturbations in the constant-speed case. Note that we cannot apply the stability result in \cite{HPW} directly because of the first-order term in our equation.

By direct computation, we check that \eqref{e:perturb} implies the following energy conservation for $\vp(t)$: if $\vp(0) = \vp^{in}$, then
\begin{equation}\label{e:conserve}
\mathcal E(\vp(t)) := \int p \vp_2^2(t) + \langle \mathcal L_K \vp_1(t),\vp(t)\rangle_p + 2 \int p K\vp_1^3(t) + \frac 1 2 \int p \vp_1^4(t) = \mathcal E(\vp^{in}),
\end{equation}
for all $t$ such that $\vp(t)$ exists in the energy space.

Next, we prove the following:
\begin{lemma}\label{l:orbit}
If $\delta>0$ is sufficiently small, then there exists $c_0>0$ such that
\[\langle \mathcal L_K \vp_1,\vp_1\rangle_p \geq c_0\|\vp_1\|_{H^1},\]
for all odd functions $\vp_1\in H^1(\R)$.
\end{lemma}
\begin{proof}
By the spectral properties of $\mathcal L_K$ and the oddness of $\vp_1$, we have 
\[\langle \mathcal L_K \vp_1,\vp_1\rangle_p \geq  \lambda_1 \|\vp_1\|_{L^2},\]
where $\lambda_1\geq \frac 3 2 -C\delta$. Next, since $(1-K^2) \leq 1$,
\begin{align*}
\langle \mathcal L_K \vp_1,\vp_1\rangle_p &= \int p (\partial_y \vp_1)^2 + 2\int p (\vp_1)^2 - 3\int p (1-K^2)(\vp_1)^2\\
&\geq \int p (\partial_y \vp_1)^2 + \frac 5 7 \int p (\vp_1)^2 - \frac {12} 7\int p (1-K^2)(\vp_1)^2.
\end{align*}
Taking $\frac 4 7$ times the first equality and subtracting it from the second line, we have
\[\langle \mathcal L_K \vp_1,\vp_1\rangle_p \geq \frac 3 7 \int p (\partial_y \vp_1)^2  - \frac 3 7 \int p (\vp_1)^2 + \frac 4 7 \langle \mathcal L_K \vp_1,\vp_1\rangle_p \geq c_0 \|\vp_1\|_{H^1},\]
since $|p(y)| \geq 1-C\delta$.
\end{proof}
With this lemma, we can prove the orbital stability of $K$ with respect to odd perturbations:
\begin{proposition}\label{p:orbit}
For $\delta$ sufficiently small, there exist $C>0$ and $\eps_0>0$, depending on $\delta$, such that for any $\eps\in (0,\eps_0)$ and any $\vp^{in}\in H^1\times L^2$ with $\|\vp^{in}\|_{H^1\times L^2} < \eps$, the solution $\vp$ to \eqref{e:perturb} with $b$ satisfying \eqref{e:decay} and with initial data $\vp(0) = \vp^{in}$ exists in $H^1\times L^2$ for all $t\in\R$ and satisfies
\[\forall t\in\R, \quad \|\vp(t)\|_{H^1\times L^2} < C\|\vp^{in}\|_{H^1\times L^2}.\]
\end{proposition}
\begin{proof}
By straightforward estimates, 
\[\mathcal E(\vp^{in}) \leq (1+C_1\delta) \left(\|\vp_2^{in}\|_{L^2}^2 + 2\|\vp_1^{in}\|_{H^1}^2\right) + O(\|\vp_1^{in}\|_{H^1}^3),\]
and by Lemma \ref{l:orbit},
\[\mathcal E(\vp(t)) \geq c_0\left(\|\vp_2(t)\|_{L^2}^2 + \|\vp_1(t)\|_{H^1}^2\right)  - O(\|\vp_1(t)\|_{H^1}^3).\]
But  $\mathcal E(\vp(t))= \mathcal  E(\vp^{in})$, by \eqref{e:conserve}.
\end{proof}

Next, we decompose our solution $\vp$ based on the spectrum of $\mathcal L_K$. In the constant-speed case, one has $K=H$, and our decomposition will reduce to the one in \cite{KMM}. Let $\bar Y_1$ be the eigenfunction satisfying $\mathcal L_K \bar Y_1 = \mu^2 \bar Y_1$, with $\mu = \sqrt{\lambda_1}$. We decompose the solution $\vp$ to \eqref{e:perturb} as follows: Define
\[z_1(t) := \langle \vp_1(t),\bar Y_1\rangle_p, \quad z_2(t) := \frac 1 \mu \langle \vp_2(t),\bar Y_1\rangle_p,\]
\[u_1(t) := \vp_1(t) - z_1(t) \bar Y_1, \quad u_2(t) := \vp_2(t) - \mu z_2(t)\bar Y_1.\]
We have $\langle u_1(t),\bar Y_1\rangle_p = \langle u_2(t),\bar Y_1\rangle_p = 0$ for all $t\in \R$. Set $z(t) := (z_1(t),z_2(t))$ and $u(t) := (u_1(t),u_2(t))$. Finally, define
\[|z|^2(t) := z_1^2(t) + z_2^2(t), \quad \alpha(t) := z_1^2(t) - z_2^2(t), \quad \beta(t) := 2z_1(t)z_2(t).\]
By \eqref{e:perturb}, we have 
\begin{equation}\label{e:z}
\begin{cases} \dot z_1 &= \mu z_2\\
\dot z_2 &= -\mu z_1  -\dfrac 1 \mu \langle 3K\vp_1^2+\vp_1^3,\bar Y_1\rangle_p.\end{cases}
\end{equation}
We have 
\begin{equation}\label{e:alphabeta}
\begin{cases} \dot \alpha = 2\mu \beta + F_\alpha,\\
\dot \beta = -2\mu \alpha + F_\beta,\end{cases}
\end{equation}
with 
\begin{align*}
F_\alpha &= \dfrac 2 \mu z_2\langle 3K\vp_1^2 + \vp_1^3,\bar Y_1\rangle_p,\\
F_\beta &=  -\dfrac 2 \mu z_1\langle 3K\vp_1^2 + \vp_1^3,\bar Y_1\rangle_p,
\end{align*}
and 
\[\frac d {dt}(|z|^2) = -F_\alpha.\]
Next, \eqref{e:perturb} implies that $u(t)$ satisfies
\begin{equation}
\begin{cases} \dot u_1 &= u_2\\
\dot u_1 &= -\mathcal L_K u_1 - 2z_1^2 \bar f + F_u, \end{cases}
\end{equation}
where 
\[F_u = -\left[ 3K(u_1^2+2u_1z_1\bar Y_1) + \vp_1^3 - \langle 3K(u_1^2 + 2u_1z_1 \bar Y_1) + \vp_1^3,\bar Y_1\rangle_p \bar Y_1\right],\]
and $\bar f = \lambda_1 (K\bar Y_1^2 - \langle K\bar Y_1^2,\bar Y_1\rangle_p \bar Y_1)$ is an odd Schwartz function satisfying $\langle \bar f, \bar Y_1\rangle_p = 0$. Since $\bar Y_1$ and $\bar Y_1'$ decay at the rate $e^{-|y|/\sqrt 2}$, $\bar f$ and $\bar f'$ have the same decay, i.e. $|\bar f|+|\bar f'|\lesssim e^{-|y|/\sqrt 2}$ as $y\rightarrow\infty$.

It will be useful to replace the term $z_1^2\bar f$ with a term involving only $\alpha$. Let $q$ be the odd solution to $\mathcal L_K q = f$. Using the methods of Sections \ref{s:stat} and \ref{s:spec}, it is straightforward to show that $q$ exists uniquely in $H^1(\R)$ and satisfies $|q(y)| + |q'(y)|\lesssim e^{-y/\sqrt 2}$. We make the change of unknown
\begin{align*}
v_1(t,y) &:= u_1(t,y) + |z|^2(t)q(y),\\
v_2(t,y) &:= u_2(t,y)
\end{align*}
Now the system becomes
\begin{equation}\label{e:v}
\begin{cases}\dot v_1 &= v_2 + F_1,\\
\dot v_2 &= -\mathcal L_K v_1 - \alpha \bar f + F_2,\end{cases}
\end{equation}
where $F_1 = -qF_\alpha$ and $F_2 = F_u$. We have $0 = \langle \bar f, \bar Y_1\rangle_p = \langle \mathcal L_K q,\bar Y_1\rangle_p = \langle q,\mathcal L_K \bar Y_1\rangle_p = \mu^2\langle q,\bar Y_1\rangle_p$, which implies $\langle v_1,\bar Y_1\rangle_p = \langle v_2,\bar Y_1\rangle_p = 0$. 

The terms $F_\alpha$, $F_\beta$, $F_1$, and $F_2$ are regarded as error terms, and will be dealt with in Section \ref{s:conclusion}.
\section{Virial arguments}\label{s:vir}

Here we analyze the system in $(v_1,v_2,\alpha,\beta)$ given by \eqref{e:v} and \eqref{e:alphabeta}.
%\begin{equation}
%\begin{cases} \dot v_1 &= v_2 + F_1,\\
%\dot v_2 &= -\mathcal L_K v_1 - \alpha \bar f + F_2,\\
%\dot \alpha &= 2\mu\beta + F_\alpha,\\
%\dot \beta &= 2\mu\alpha + F_\beta.\end{cases}
%\end{equation}
Following \cite{KMM}, we define
\[\mathcal I := \int \psi(\partial_y v_1)v_2 + \frac 1 2 \int \psi'v_1v_2,\]
with $\psi = 8 \sqrt 2 \tanh(y/8\sqrt 2)$ and 
\[\mathcal J := \alpha \int v_2 \bar g - 2\mu \beta \int v_1 \bar g,\]
with $\bar g$ to be chosen later. Differentiating and using the system \eqref{e:v} for $v_1$ and $v_2$, we have
\begin{align*}
\frac d {dt} \int \psi(\partial_y v_1)v_2 &= \int \psi(\partial_y \dot v_1)v_2 + \int \psi(\partial_y v_1) \dot v_2\\
&= \int \psi(\partial_y v_2)v_2 + \int \psi (\partial_y v_1)(\partial_y^2 v_1 + b\partial_yv_1 - 2v_1  + 3(1-K^2)v_1)\\
&\quad - \alpha \int \psi (\partial_y v_1)\bar f + \int \psi((\partial_y F_1)v_2 + (\partial_y v_1)F_2)\\
&= -\frac 1 2 \int \psi' (v_2^2 + (\partial_y v_1)^2 - 2v_1^2) + \int \psi b(\partial_y v_1)^2 - \frac 3 2 \int (\psi(1-K^2))'v_1^2\\
&\quad + \alpha \int v_1(\psi \bar f)' + \int \psi((\partial_y F_1)v_2 + (\partial_y v_1)F_2),
\end{align*}
and
\begin{align}\label{e:psiprime}
\frac d {dt} \int \psi' v_1v_2 &= \int \psi' \dot v_1 v_2 + \int \psi' v_1 \dot v_2\nonumber\\
&= \int \psi'v_2^2 + \int \psi'v_1(\partial_y^2 + b\partial_yv_1 - 2v_1 + 3(1-K^2)v_1)\nonumber\\
&\quad  - \alpha \int \psi'v_1\bar f + \int \psi'(F_1v_2 + v_1 F_2)\nonumber\\
&= \int \psi'(v_2^2 - (\partial_y v_1)^2 - 2v_1^2) + \frac 1 2 \int \psi'''v_1^2 - \frac 1 2 \int (\psi'b)'v_1^2 \nonumber\\
&\quad + 3\int \psi'(1-K^2)v_1^2- \alpha \int \psi' v_1 \bar f + \int \psi'(F_1v_2 + v_1 F_2),
\end{align}
which leads to
\begin{align*}
\frac d {dt} \mathcal I &= -\tilde{\mathcal B}(v_1) + \alpha \int v_1 (\psi \bar f' + \frac 1 2 \psi' \bar f) + \int \psi b (\partial_y v_1)^2 - \frac 1 4 \int (\psi'b)'v_1^2\\
&\quad + \int v_2 (\psi \partial_y F_1 + \frac 1 2 \psi'F_1) - \int v_1(\psi \partial_y F_2 + \frac 1 2 \psi'F_2)
\end{align*}
where
\[\tilde{\mathcal B}(v_1) := \int \psi'(\partial_y v_1)^2 - \frac 1 4 \int \psi'''v_1^2 - 3 \int \psi K K'v_1^2.\]
Differentiating $\mathcal J$, we have
\begin{align*}
\frac d{dt}\mathcal J &= \dot \alpha \int v_2 \bar g + \alpha\int \dot v_2 \bar g - 2\mu \dot \beta\int v_2 \bar g - 2\mu \beta \int \bar g\dot v_1\\
&= \alpha\int \bar g (-\mathcal L_K v_1  + 4\mu^2) - \alpha^2 \int \bar f \bar g \\
&\quad + F_\alpha \int v_2 \bar g - 2\mu F_\beta \int v_1 \bar g - 2\mu \beta \int \bar g F_1 + \alpha \int \bar g F_2.
\end{align*}
Note that
\[\int \bar g(-\mathcal L_K v_1 + 4\mu^2 v_1) = \int p \, \frac {\bar g} p \left(-\mathcal L_K v_1 + 4\mu^2 v_1\right) = \int p v_1 \left(-\mathcal L_k \left(\frac {\bar g} p\right) + 4\mu^2 \frac {\bar g} p \right), \]
so combining these calculations, we obtain
\[\frac d {dt}(\mathcal I + \mathcal J) = -\tilde{\mathcal D}(v_1,\alpha) + \mathcal R_{\tilde{\mathcal D}},\]
where
\begin{align}\label{e:D}
\tilde{\mathcal D}(v_1,\alpha) &:= \tilde{\mathcal B}(v_1) - \int \psi b (\partial_y v_1)^2 + \frac 1 4 \int (\psi'b)'v_1^2\\
&\quad - \alpha\int p v_1\left(\frac{\psi \bar f' + \frac 1 2 \psi' \bar f} p - \mathcal L_K \left( \frac {\bar g} p\right) + 4\mu^2 \frac {\bar g} p\right) + \alpha^2\int \bar f \bar g,\nonumber
\end{align}
and
\begin{equation}\label{e:remainder}
\mathcal R_{\tilde{\mathcal D}} := \int \bar g(\alpha F_2 - 2\mu\beta F_1) + \int v_2(\psi \partial_y F_1 + \frac 1 2 \psi' F_1 + \bar gF_\alpha) - \int v_1(\psi \partial_y F_2 + \frac 1 2 \psi' F_2 + 2\mu \bar g F_\beta).
\end{equation}

We will choose $\bar g$ in order to simplify $\tilde{\mathcal D}$ considerably.  In \cite{KMM}, the authors chose $g$ in the functional $\mathcal J$ by solving the equation 
\begin{equation}\label{e:Lminus6}
(\mathcal L - 6)g = \psi f' + \left(a+\frac 1 2\right) \psi' f,
\end{equation}
where $f = \frac 3 2 (H Y_1^2 - \langle H Y_1^2, Y_1\rangle Y_1)$, and the constant 
\begin{equation}\label{e:a}
a:= -\dfrac {\langle \psi f' + \frac 1 2 \psi' f, \mbox{Im}(k)\rangle}{\langle \psi' f, \mbox{Im}(k)\rangle} \approx 0.687271,
\end{equation}
where $k$ is the function defined by \eqref{e:kdef} below. The value of $a$ was found numerically. We quote a lemma from \cite{KMM} that allows one to solve \eqref{e:Lminus6}. The form of the function $k$ in Lemma \ref{l:k} was originally found by Segur \cite{S}.
\begin{lemma}[{\cite[Lemma 3.1]{KMM}}]\label{l:k}
\textup{(a)} Let $F\in L^1(\R)\cap C^1(\R)$ be a real-valued function. The function $G\in L^\infty(\R)\cap C^2(\R)$ defined by
\[G(y) = \frac 1 {12} \mbox{\textup{Im}}\left(k(y) \int_{-\infty}^y \bar k F + \bar k(y)\int_y^\infty kF\right),\]
where
\begin{equation}\label{e:kdef}
k(y) = e^{2iy}\left(1 + \frac 1 2 \sech ^2\left(\frac y{\sqrt 2}\right)+ i\sqrt 2 \tanh\left(\frac y{\sqrt 2}\right)\right),
\end{equation}
and $\bar k$ is the complex conjugate of $k$, satisfies
\[(-\mathcal L + 6)G = F.\]
\textup{(b)} Assume in addition that $F\in \mathcal S(\R)$, the class of Schwartz functions. Then,
\[G\in \mathcal S(\R) \quad \Longleftrightarrow \quad \langle k,F\rangle = 0.\]
\end{lemma}
Since $k(-y) = \bar k(y)$, if $F$ is odd then $G$ is odd as well, and the orthogonality condition in Lemma \ref{l:k}(b) reduces to $\langle \mbox{Im } k,F\rangle = 0$.

In our case, we would like to find $\bar h$ solving
\begin{equation}\label{e:LKminus}
\mathcal L_K \bar h - 4\mu^2 \bar h = \dfrac 1 p (\psi \bar f' + (a_0 + \frac 1 2) \psi' \bar f),
\end{equation}
and let $\bar g = p \bar h$. The constant $a_0$ is defined by
\[a_0 := -\dfrac {\langle (\psi \bar f' + \frac 1 2 \psi' \bar f)/p, \mbox{Im}(k)\rangle}{\langle \psi'\bar f/p, \mbox{Im}(k)\rangle}.\]% = \dfrac {\langle \psi \bar f' + \frac 1 2 \psi' \bar f, \mbox{Im}(k)\rangle}{\langle \psi'\bar f, \mbox{Im}(k)\rangle}.\] 
From \cite{KMM}, we have $\langle \psi' f, \mbox{Im}(k)\rangle \approx -0.327$. From our Theorem \ref{t:spectrum}, we have 
\[|\bar f(y) - f(y)| \lesssim \delta e^{-|y|/\sqrt 2},\quad |\bar f'(y) - f'(y)|\lesssim \delta e^{-|y|/\sqrt 2},\]
which implies $\langle \psi'\bar f/p, \mbox{Im}(k)\rangle < -0.3$ for $\delta$ sufficiently small. (Recall $\|p-1\|_{L^\infty}\lesssim \delta$.) We also claim that $|a-a_0| \lesssim \delta$, where $a$ is defined by \eqref{e:a}. Indeed, we have
\[a-a_0 = \frac{a \langle \psi' (\bar f /p - f), \mbox{Im}(k)\rangle  + \langle (\psi (\bar f'/p - f') + \frac 1 2 \psi' (\bar f/p - f),\mbox{Im}(k)\rangle}{\langle \psi' \bar f /p, \mbox{Im}(k)\rangle}, \]
so that $|a-a_0|\leq C(\|\bar f / p - f\|_{L^\infty} + \|\bar f'/p - f\|) \lesssim \delta$. We conclude $a_0 > 0$ for $\delta$ small enough. This allows us to solve \eqref{e:LKminus}:
\begin{lemma}\label{l:barh}
There exists an odd $\bar h \in \mathcal S(\R)$ solving \eqref{e:LKminus}. Furthermore, $\bar h$ satisfies
\[|\bar h(y)|+|\bar h'(y)|\lesssim e^{-|y|/\sqrt 2},\]
and $\|g-\bar h\|_{L^\infty}\lesssim \delta$, where $g$ is the unique solution of \eqref{e:Lminus6}.
\end{lemma}
\begin{proof}
Let $\ell = (\psi \bar f' + (a_0 + \frac 1 2) \psi'\bar f)/p$, and define
\[h (y) = \frac 1 {12} \mbox{Im}\left(k(y) \int_{-\infty}^y \bar k \ell + \bar k(y)\int_y^\infty k\ell\right).\]
By Lemma \ref{l:k}, $h$ solves $\mathcal L h - 6 h = \ell$, Since $\ell$ is odd and $k(-y) = \bar k(y)$, we have that $h$ is odd. By our choice of $a_0$, we have $\langle \ell, \mbox{Im}(k)\rangle = 0$, which implies $h$ is Schwarz class. In fact, the decay of $\bar f'$ and $\bar f$, and the explicit formula for $k$, imply that $|h(y)| + |h'(y)|\lesssim e^{-|y|/\sqrt 2}$. Next, we set up an integral equation for the difference between $h$ and $\bar h$, as above. If $\bar h$ satisfies $\mathcal L_K \bar h - 4\mu^2 \bar h = \ell$, then $\eta = \bar h - h$ satisfies
\begin{equation}\label{e:4mu2}
(\mathcal L - 4\mu^2)\eta  = bh' - dh  + (4\mu^2-6)h + b\eta' - d\eta.
\end{equation}
Recall $|\mu^2 - \frac 3 2 |\lesssim \delta$. We construct a Green's function for $\mathcal L - 4\mu^2$ on $[0,\infty)$ using a modification of the function $k$. Let $\gamma = \sqrt{4\mu^2-2}$, $c_1 = \dfrac {3}{4\mu^2+1}$, $c_2 = \dfrac{3\gamma}{8\mu^2+2}$, and
\[k^\circ(y) = e^{i\gamma y}\left(1 + \frac 1 2 c_1\sech^2\left(\frac y{\sqrt 2}\right)+ ic_2\sqrt 2 \tanh\left(\frac y{\sqrt 2}\right)\right).\]
It can be checked by direct computation that $\mathcal L k^\circ - 4\mu^2 k^\circ = 0$. The Wronskian $W(\mbox{Re } k^\circ,\mbox{Im } k^\circ)$ is given by the constant $c_0:= (1+\frac{c_1}2)(c_1 + \gamma(1+\frac{c_1}2))$. Define the Green's function
\[G_\mu(y,w) = \begin{cases}\mbox{Im } k^\circ(y)\mbox{Re } k^\circ(w)/c_0, &0\leq y < w,\\
				       \mbox{Re }k^\circ(y)\mbox{Im } k^\circ(w)/c_0, &0\leq w < y.\end{cases}\]
Then the ODE \eqref{e:4mu2} is equivalent to the integral equation
\[\eta = \int_0^\infty G_\mu(y,w) (	bh' - dh  + (4\mu^2-6)h)(w)\dd w + \int_0^\infty G_\mu(y,w)(b\eta' - d\eta)(w)\dd w.\]
The first integral on the right-hand side converges because of the decay of $h$, so we can integrate by parts in the second integral and apply Lemma \ref{l:Fredholm}, using properties of $b$ and $d$.  As above, formula \eqref{e:series} and the decay of $G_\mu$ imply the solution $\eta$ satisfies $|\eta(y)|+|\eta'(y)|\lesssim e^{-|y|/\sqrt 2}$. Since $\eta(0)=0$, we can extend it by oddness to obtain $\bar h = h + \eta$, an exponentially decaying solution to $\mathcal L_K \bar h - 4\mu^2\bar h = \ell$ on the real line. 

For the last claim, let $\tilde \ell = \psi f' + (a+ \frac 1 2) \psi'f$. Then by \eqref{e:Lminus6}, we have $(\mathcal L - 6)g = \tilde \ell$. It is clear that $|\ell(y) - \tilde \ell(y)| \lesssim \delta e^{-|y|/\sqrt 2}$, and the relationship $(\mathcal L-6)(g-h) = \tilde \ell - \ell$ implies $g-h$ can be written
\begin{align*}
|g(y)-h(y)| &= \left|\frac 1 {12} \mbox{Im}\left(k(y) \int_{-\infty}^y \bar k (\tilde \ell - \ell) + \bar k(y)\int_y^\infty k(\tilde \ell- \ell)\right)\right|,\\
&\lesssim \delta \|k\|_{L^\infty}^2 \int_{-\infty}^y e^{-|s|/\sqrt 2} \dd s \lesssim \delta.
\end{align*}
Since $\|\bar h - h\|_{L^\infty} = \|\eta\|_{L^\infty} \lesssim \delta$, we conclude $\|g - \bar h\|_{L^\infty} \lesssim \delta$.
\end{proof}

With $\bar g = p \bar h$, \eqref{e:D} simplifies to
\[\tilde{\mathcal D}(v_1,\alpha) = \tilde{\mathcal B}(v_1) - \int \psi b (\partial_y v_1)^2 + \frac 1 4 \int (\psi'b)'v_1^2 - a_0 \alpha\int \psi' \bar f v_1 + \alpha^2\int \bar f \bar g.\]
From $\|p - 1\|_{L^\infty} \lesssim \delta$, it is clear that $\|\bar g - g\|_{L^\infty} \leq \|\bar g- \bar h\|+\|\bar h - g\| \lesssim \delta$.

Since $\tilde{\mathcal B}$ and $\tilde{\mathcal D}$ are perturbations of forms that arise in the constant-speed case, we quote the coercivity results obtained in \cite{KMM} for the unperturbed forms $\mathcal B$ and $\mathcal D$. We work with the following weighted norms, which will be technically convenient in the later stages of the proof:
\[\|v_1\|_{H_\omega^1}^2 := \int \left(|\partial_yv_1|^2 + v_1^2\right)\,\sech\left(\frac y {2\sqrt 2}\right)\dd y, \quad \|v_2\|_{L_\omega^2}^2 := \int v_2^2\,\sech\left(\frac y {2\sqrt 2}\right)\dd y,\]
and
\begin{equation}\label{e:weighted}
\|v\|_{H_\omega^1\times L_\omega^2}^2 := \|v_1\|_{H_\omega^1}^2 + \|v_2\|_{L_\omega^2}^2.
\end{equation}
It is also convenient to work with the auxiliary function $w = \zeta v_1$, where $\zeta(y) = \sqrt{\psi'(y)} = \sech\left(y/8\sqrt 2\right)$. It can be shown by direct computation that
\begin{equation}\label{e:zeta}
\|v_1\|_{H_\omega^1} \lesssim \|\partial_y w\|_{L^2} = \|\partial_y(\zeta v_1)\|_{L^2},
\end{equation}
see \cite[Proposition 5.1]{KMM} for the proof.
\begin{lemma}\label{l:BD}
\textup{(a) (\cite[Lemma 4.1]{KMM})} Define the quadratic form
\[\mathcal B(v) := \int \psi'(\partial_y v)^2 - \frac 1 4 \int \psi'''v^2 - 3 \int \psi H H'v^2.\]
There exists $\kappa_1 >0$ such that, for any odd function $v\in H^1$,
\begin{equation}\label{e:Btilde}
\langle v, Y_1 \rangle = 0 \quad \implies \quad \mathcal B(v) \geq \kappa_1\|\partial_y(\zeta v)\|_{L^2}^2,
\end{equation}
where $Y_1$ is the eigenfunction satisfying $\mathcal L Y_1 = \frac 3 2 Y_1$.

\textup{(b) (\cite[Lemma 4.2]{KMM})} Define the bilinear form
\[\mathcal D(v,\alpha) = \mathcal B(v) -\alpha a\int \psi' fv + \alpha^2\int fg,\]
with $a,f,$ and $g$ as defined above. There exists $\kappa_2>0$ such that for every odd $v\in H_\omega^1$,
\begin{equation}\label{e:tildeD}
\langle v, Y_1 \rangle = 0 \quad \implies \mathcal D(v,\alpha) \geq \kappa_2(\alpha^2 + \|\partial_y(\zeta v)\|_{L^2}^2).
\end{equation}
\end{lemma}

First, we extend Lemma \ref{l:BD}(a) to the perturbed quadratic form $\tilde{\mathcal B}$:
\begin{lemma}\label{l:B}
There exists $\kappa >0$ such that, for any odd function $v_1 \in H_\omega^1$, 
\[ \langle v_1, \bar Y_1 \rangle_p = 0 \quad \implies \quad \tilde{\mathcal B}(v_1) \geq \kappa \|\partial_y (\zeta v_1)\|_{L^2}^2.\]
\end{lemma}
\begin{proof}
To apply Lemma \ref{l:BD}(a), we decompose $v_1 = \tilde v_1 + \langle v_1, Y_1 \rangle Y_1$, so that $\langle \tilde v_1, Y_1\rangle = 0$. 
By the definition of $\tilde{\mathcal B}$, we have
\begin{align}\label{e:B}
\tilde{\mathcal B}(v_1) &= \mathcal B(\tilde v_1) -3\int \psi (H_\delta K' + H H_\delta') (\tilde v_1)^2\nonumber\\
&\quad + \int \psi' \left[ \langle v_1,  Y_1\rangle^2 (\partial_y Y_1)^2 + 2\partial_y \tilde v_1 \langle v_1,Y_1\rangle \partial_y Y_1\right]\\
&\quad -\int \left(\frac 1 4 \psi''' + 3 \psi K K'\right) \left[ \langle v_1,  Y_1\rangle^2  Y_1^2 + 2\tilde v_1 \langle v_1, Y_1\rangle  Y_1\right].\nonumber
\end{align}
From Theorem \ref{t:spectrum}, we have $\|Y_1 - \bar Y_1\|_{L^\infty} \lesssim \delta$. We conclude from $\langle v_1,\bar Y_1\rangle_p = 0$ and $\|p(y)-1\|_{L^\infty}\lesssim \delta$ that 
\begin{equation}\label{e:tildeY1}
|\langle  v_1, Y_1\rangle| \leq |\langle v_1,Y_1-\bar Y_1\rangle_p|+|\langle v_1,Y_1(1-p)\rangle| \lesssim \delta \|v_1\|_{L^2}.
\end{equation} 
Since $|H_\delta K' + H H_\delta'| \lesssim \delta e^{-\sqrt 2 |y|}$ and $|\langle v_1, Y_1\rangle| \leq \| v_1\|_{H_\omega^1}\lesssim \|\partial_y(\zeta v_1)\|_{L^2}$ (by the exponential decay of $Y_1$), we conclude from Lemma \ref{l:BD}(a), \eqref{e:B}, \eqref{e:tildeY1}, and the decay of $\psi'$ and $Y_1$ that
\[ \tilde{\mathcal B}(v_1) \geq (\kappa_0- C\delta)\|\partial_y(\zeta\tilde v_1)\|_{L^2}^2 - C\delta\|\partial_y(\zeta v_1)\|_{L^2}^2.\]
Finally, observe that for $\delta$ sufficiently small, \eqref{e:tildeY1} implies $\|\partial_y(\zeta\tilde v_1)\|_{L^2} \geq \frac 1 2 \|\partial_y(\zeta v_1)\|_{L^2}$. Indeed, we have $\partial_y(\zeta v_1) - \partial_y(\zeta \tilde v_1) = \langle v_1, Y_1\rangle \partial_y(\zeta Y_1)$, and $\partial_y(\zeta Y_1)$ is an explicit function in $L^2$. We conclude $\tilde{\mathcal B}(v_1) \geq \kappa \|\partial_y(\zeta v_1)\|_{L^2}$.
\end{proof}

We are now ready to prove the coercivity of $\tilde{\mathcal D}$:
\begin{lemma}\label{l:D}
There exists $\kappa>0$ such that for any odd $v_1\in H_\omega^1$ such that $\langle v_1, \bar Y_1\rangle_p = 0$, 
\begin{equation}\label{e:Dcoerce}
\tilde{\mathcal D}(v_1,\alpha) \geq \kappa\left(\alpha^2 + \|\partial_y (\zeta v_1)\|_{L^2}^2\right).
\end{equation}
\end{lemma}
\begin{proof}
Proceeding similarly to the proof of Lemma \ref{l:B}, we write $v_1 = \tilde v_1 + \langle v_1, Y_1 \rangle Y_1$, with $|\langle v_1,Y_1\rangle| \lesssim \delta$.
Writing 
\begin{align*}
\tilde{\mathcal D}(v_1,\alpha) &= \mathcal D(\tilde v_1,\alpha) + (\tilde{\mathcal B}(v_1) - \mathcal B(\tilde v_1)  ) -(a_0\alpha\int \psi' \bar fv_1 - a\alpha \int \psi' f \tilde v_1)\\ 
&+ \alpha^2\left(\int \bar f \bar g - \int f g\right) - \int \psi b(\partial_y v_1)^2 + \frac 1 4 (\psi'b)'v_1^2,
\end{align*}
from the proof of Lemma \ref{l:B}, we have
\[|\tilde{\mathcal B}(v_1) - \mathcal B(\tilde v_1)| \lesssim \delta \|v_1\|_{H_\omega^1}.\]
Because $|\bar f - f|\lesssim \delta e^{-|y|/\sqrt 2}$ and $|a_0 - a|\lesssim \delta$, the next term
\[\left| a_0\alpha\int \psi' \bar fv_1 - a\alpha \int \psi' f \tilde v_1\right| \lesssim \delta \alpha\|v_1\|_{H_\omega^1} \lesssim \delta(\alpha^2 + \|v_1\|_{H_\omega^1}^2).\]
Since $\bar f$ and $\bar g$ are $\delta$-close to $f$ and $g$, we have
\[\alpha^2\left|\int \bar f \bar g - \int f g\right| \lesssim \delta \alpha^2.\]
%\[\mathcal B(v_1) \gtrsim \kappa_1\|v_1\|_{H^1}^2.\]
Finally, the bound \eqref{e:decay} clearly implies
\[\left|\int \psi b (\partial_y v_1)^2 - \frac 1 4 \int (\psi'b)'v_1^2\right| \lesssim \delta \|v_1\|_{H_\omega^1}^2.\]
Combining these bounds with \eqref{e:zeta} and Lemma \ref{l:BD}(b), we obtain \eqref{e:Dcoerce} for sufficiently small $\delta$.
\end{proof}

\section{Conclusion of the Proof}\label{s:conclusion}

%In this section...
Let $\vp^{in} \in H^1\times L^2$ be odd and satisfy $\|\vp^{in}\|_{H^1\times L^2} < \eps$ for $\eps>0$ a small number to be chosen. Proposition \ref{p:orbit} implies that the solution $\vp$ of \eqref{e:perturb} with initial data $\vp^{in}$ exists in $H^1\times L^2$ and
\[\|\vp(t)\|_{H^1\times L^2} \lesssim \eps,\]
for all $t\in\R$. By the spectral decomposition of Section \ref{s:orbit}, this implies
\begin{equation}\label{e:small}
\|u(t)\|_{H^1\times L^2} + \|v(t)\|_{H^1\times L^2} + \|u_1(t)\|_{L^\infty} + \|v_1(t)\|_{L^\infty} + |z(t)| \lesssim \eps,
\end{equation}
for all $t\in \R$. 

The proof of Theorem \ref{t:main} relies on the following fact, whose proof closely mirrors the proof of Proposition 5.1 in \cite{KMM}.
\begin{proposition} For $z(t) = (z_1(t),z_2(t))$ satisfying \eqref{e:z} and $v(t) = (v_1(t),v_2(t))$ satisfying \eqref{e:v}, one has
\begin{equation}\label{e:time}
\int_\R \left(|z(t)|^4 + \|v(t)\|_{H_\omega^1\times L_\omega^2}^2\right) \dd t \lesssim \eps^2.
\end{equation}
\end{proposition}
\begin{proof}
With $\alpha,\beta$ defined as above and satisfying \eqref{e:alphabeta}, let $\gamma(t) = \alpha(t)\beta(t)$. We will prove \eqref{e:time} as a consequence of the following three estimates:
\begin{align}
\frac d {dt} \gamma &\geq 2\mu(\beta^2-\alpha^2) - C\eps(|z(t)|^4 + \|v_1\|_{H_\omega^1}^2),\label{e:gamma}\\
-\frac d {dt}(\mathcal I + \mathcal J) &\geq \kappa (\alpha^2 + \|v_1\|_{H_\omega^1}^2) - C\eps (|z(t)|^4 + \|v_2\|_{L_\omega^2}^2),\label{e:IplusJ}\\
2 \frac d {dt} \int \sech\left(\frac y {2\sqrt 2}\right) v_1v_2 &\geq \|v_2\|_{L_\omega^2}^2 - C(|z(t)|^4 + \|v_1\|_{H_\omega^1}^2),\label{e:v1v2}
\end{align}
where $\mu, \kappa,$ and $C$ are fixed positive constants, and $w = v_1\sech(y/8\sqrt 2)$ as above.

For \eqref{e:gamma}, note that
\[\dot \gamma = \dot \alpha \beta + \alpha \dot \beta = 2\mu(\beta^2-\alpha^2) + \mathcal R_\gamma,\]
with $\mathcal R_\gamma = \beta F_\alpha + \alpha F_\beta$. Recalling that $F_\alpha = \dfrac 2 \mu z_2\langle 3K\vp_1^2 + \vp_1^3,\bar Y_1\rangle_p$ and $F_\beta =  -\dfrac 2 \mu z_1\langle 3K\vp_1^2 + \vp_1^3,\bar Y_1\rangle_p$, we substitute $\vp_1 = u_1+z_1\bar Y_1 = v_1 - |z|^2q + z_1\bar Y_1$ and use the exponential decay of $\bar Y_1$ (Theorem \ref{t:spectrum}) to obtain
\begin{equation}\label{e:FalphaFbeta}
|F_\alpha| + |F_\beta| \lesssim |z|\left(|z|^2+\|v_1\|_{L_\omega^2}^2\right).
\end{equation}
Since $|\alpha|,|\beta|\lesssim |z|^2$, \eqref{e:small} implies
\[|\mathcal R_\gamma|\lesssim |z|^3\left(|z|^2 + \|v_1\|_{L_\omega^2}^2\right)\lesssim \eps\left(|z|^4 + \|v_1\|_{L_\omega^2}^2\right).\]

To prove \eqref{e:IplusJ}, one can read off the proof of the corresponding statement in \cite[Proposition 5.1]{KMM} verbatim, with $K, \bar f,$ and $\bar g$ replacing $H, f,$ and $g$.  In particular, the coercivity of $\tilde{\mathcal D}(v,\alpha)$ (Lemma \ref{l:D}) implies it is sufficient to show
\begin{equation}\label{e:Rest}
|\mathcal R_{\tilde{\mathcal D}}| \lesssim \eps\left(|z(t)|^4 + \|\partial_y w\|_{L^2}^2 + \|v_2\|_{L_\omega^2}^2\right),
\end{equation}
and use \eqref{e:zeta}, where $\mathcal R_{\tilde{\mathcal D}}$ is given by \eqref{e:remainder}. The estimate \eqref{e:Rest} relies on the exponential decay of $\bar Y_1$ and $\bar g$ and is proven exactly as in \cite{KMM}, since the remainder $\mathcal R_{\tilde{\mathcal D}}$ is formally the same as in the constant-speed case, including the error terms $F_\alpha$, $F_\beta$, $F_1$, and $F_2$. 

We now prove \eqref{e:v1v2}. Replacing $\psi'$ with $\theta = \sech(\frac y {2\sqrt 2})$ in \eqref{e:psiprime}, we have
\begin{align}\label{e:theta}
\frac d {dt} \int \theta v_1v_2 &= \int \theta(v_2^2 - (\partial_y v_1)^2 - 2v_1^2) + \frac 1 2 \int (\theta'' +(\theta b)')v_1^2 - \int \theta b v_1\partial_y v_1\nonumber\\
&\quad + 3\int \theta(1-K^2)v_1^2 - \alpha \int \theta v_1 \bar f + \int \theta(F_1 v_2 + v_1 F_2).
\end{align}
Since $\theta'$ and $\theta''$ have the same decay as $\theta$ as $y\to\infty$, we have
\[\int \theta \left[(\partial_y v_1)^2 + (3K^2-1)v_1^2 + |bv_1\partial_yv_1|\right] + \int |(\theta b)'| v_1^2 + \frac 1 2 \int |\theta''| v_1^2 \lesssim \|v_1\|_{H_\omega^1}^2,\]
and
\[ \left|\alpha\int \theta v_1\bar f\right| \lesssim |z|^2\|v_1\|_{H_\omega^1} \lesssim |z|^4 + \|v_1\|_{H_\omega^1}^2,\]%\int \theta v_2^2 \lesssim \|v_2\|_{L_\omega^2}^2, \qquad
since $|\alpha| \lesssim |z|^2$ and $|\bar f(y)|\lesssim e^{-|y|/\sqrt 2}$. 
%The terms $ \int \psi' F_1 v_2$ and $\int \psi' F_2 v_1$ appear in $\mathcal R_{\tilde{\mathcal D}}$ (see \eqref{e:remainder}), so because $|\theta(y)| \lesssim \psi'(y)= \sech^2(y/8\sqrt 2)$, we conclude from \eqref{e:Rest} that
Recalling that $F_1 = -qF_\alpha$, \eqref{e:FalphaFbeta} implies
\begin{equation}\label{e:F1v2}
\left|\int \theta F_1 v_2\right| \lesssim  |z|\left(|z|^2 + \|v_1\|_{L_\omega^2}^2\right)\|v_2\|_{L_\omega^2} \lesssim \eps \left(|z|^4 + \|v_2\|_{L_\omega^2}^2 + \|v_1\|_{L_\omega^2}^2\right).
\end{equation} 
To deal with the term $\int \theta F_2 v_1$, recall the expression for $F_2$, written in terms of $v_1$:
\begin{align*}
F_2 &= -\left[ 3K((v_1 - |z|^2q)^2+2(v_1 - |z|^2q)z_1\bar Y_1) + (v_1 - |z|^2q+z_1\bar Y_1)^3\right.\\
&\quad \left.- \langle 3K((v_1 - |z|^2q)^2 + 2(v_1 - |z|^2q)z_1 \bar Y_1) + (v_1 - |z|^2q+z_1\bar Y_1)^3,\bar Y_1\rangle_p \bar Y_1\right].
\end{align*}
From the decay of $q$ and $\bar Y_1$, it is straightforward to obtain
\begin{equation}\label{e:F2v1}
\left|\int \theta F_2 v_1\right| \lesssim \|v_1\|_{L_\omega^2}^2 + |z|^3\|v_1\|_{L_\omega^2} + |z|^4  \lesssim |z|^4 + \|v_1\|_{L_\omega^2}^2.
\end{equation}
With these estimates, \eqref{e:theta} implies
\[\frac d {dt} \int \theta v_1v_2 \geq \frac 1 2 \|v_2\|_{L_\omega^2}^2 - C\left(|z|^4+ \|v_1\|_{H_\omega^1}^2\right).\]
%which implies \eqref{e:v1v2}.
% since $\|v_1\|_{H_\omega^1}\lesssim \|\partial_y w\|_{L^2}$.

To prove the proposition, let
\[\mathcal K := \frac \kappa {4\mu} \gamma - (\mathcal I + \mathcal J) + 2\sigma\int\sech\left(\frac y {2\sqrt 2}\right) v_1v_2,\]
with $\sigma>0$ to be chosen. Differentiating and using \eqref{e:gamma}, \eqref{e:IplusJ}, and \eqref{e:v1v2}, we have
\[\frac d {dt} \mathcal K \geq \frac {\kappa}2(\alpha^2+\beta^2) + \kappa\|v_1\|_{H_\omega^1}^2 + \sigma\|v_2\|_{L_\omega^2}^2 - C(\sigma + \eps)\left(|z(t)|^4 + \|v_1\|_{H_\omega^1}^2\right) - C\eps\|v_2\|_{L_\omega^2}^2.\]
Since $\alpha^2+\beta^2 = |z|^4$, we can choose $\sigma>0$ small enough and then $\eps>0$ small enough that
\begin{equation}\label{e:dKdt}
\frac d {dt} \mathcal K \gtrsim |z(t)|^4 + \|v_2\|_{L_\omega^2}^2 + \|v_1\|_{H_\omega^1}^2 \gtrsim |z(t)|^4 + \|v\|_{H_\omega^1\times L_\omega^2}^2.
\end{equation}
Next, straightforward integral estimates applied to the expressions for $\mathcal I, \mathcal J$, and $\gamma$ imply
\[\left|\mathcal K(t) \right|\lesssim \|v(t)\|_{H^1\times L^2}^2 + |z(t)|^4 \lesssim \eps^2,\]
uniformly in $t\in \R$, where the last inequality follows from \eqref{e:small}. We integrate \eqref{e:dKdt} on $[-t_0,t_0]$ and send $t_0\to\infty$ to obtain \eqref{e:time}.
\end{proof}

We are now in a position to prove our main result.

\begin{proof}[Proof of Theorem \textup{\ref{t:main}}] Let
\[\mathcal H := \int \left((\partial_y v_1)^2 + 2v_1^2 + v_2^2\right)\sech\left(\frac y{2\sqrt 2}\right).\]
With $\theta(y) = \sech(y/2\sqrt 2)$ as above, we differentiate $\mathcal H$:
\begin{align*}
\dot{\mathcal H} &= 2\int \theta (\partial_y \dot v_1 \partial_y v_1 + 2\dot v_1 v_1 + \dot v_2 v_2)\\
&= 2\int \theta[\partial_y v_2 \partial_y v_1 + 2v_2 v_1 -(\mathcal L_K v_1)v_2 - \alpha \bar f v_2 +\partial_y F_1 \partial_y v_1 + 2F_1v_1 + F_2 v_2]\\
&= -2\int \theta'v_2\partial_y v_1 + 2 \int \theta[3(1-K^2)v_1v_2 - \alpha \bar f v_2 + bv_2 \partial_y v_1 - dv_1v_2]\\
&\quad + 2 \int \theta(\partial_y F_1 \partial_y v_1 + 2F_1v_1 + F_2 v_2). 
\end{align*}
Note that
\[\left|\int \theta' v_2(\partial_y v_1) + \int \theta bv_2\partial_y v_1 - \int \theta d v_1 v_2\right|\lesssim \int\theta [(\partial_y v_1)^2 + v_1^2 + v_2^2] .\]
In a similar manner to \eqref{e:F1v2} and \eqref{e:F2v1}, one can show
\[\int \theta [\partial_y F_1 \partial_y v_1 + 2F_1 v_1 + F_2 v_2] \lesssim |z|^4 + \|v\|_{H_\omega^1\times L_\omega^2}^2,\]
and we conclude
\begin{equation}\label{e:dot}
|\dot{\mathcal H}| \lesssim |z(t)|^4 + \|v(t)\|_{H_\omega^1\times L_\omega^2}^2.
\end{equation}
By the orbital stability, there exists a sequence $t_n\rightarrow \infty$ with $\mathcal H(t_n) + z(t_n) \rightarrow 0$. Given $t\in \R$, integrate \eqref{e:dot} from $t$ to $t_n$ and pass to the limit as $n\rightarrow \infty$ to obtain
\[\mathcal H(t) \lesssim\int_t^\infty \left(|z(t)|^4+\|v(t)\|_{H_\omega^1\times L_\omega^2}^2\right)\dd t.\]
Combined with \eqref{e:time}, this implies $\lim_{t\rightarrow\infty} \mathcal H(t) = 0$. By a similar argument, $\lim_{t\rightarrow-\infty} \mathcal H(t) = 0$. Note that by \eqref{e:FalphaFbeta},
\[\left|\frac d {dt}|z|^4\right| = 2|\alpha F_\alpha + \beta F_\beta| \lesssim  |z|^3\left(|z|^2+\|v_1\|_{L_\omega^2}^2\right) \lesssim |z|^4 + \|v_1\|_{L_\omega^2}^2,\]
so we can integrate in time as above and conclude $z(t)\rightarrow 0$ as $t\rightarrow \pm\infty$. Since $u_1 = v_1 - q|z|^2$, we have $\lim_{t\rightarrow\pm\infty} \|u(t)\|_{H^1(I)\times L^2(I)} = 0$ for any bounded interval $I$, as desired.
\end{proof}

\end{document}